\documentclass[a4paper,11pt]{amsart}
\usepackage[all,arc]{xy}
\usepackage[all]{xy}
\usepackage{latexsym}
\usepackage{amssymb}
\usepackage{amsmath}
\usepackage{amsthm}
\newtheorem{definition}{Definition}[section]
\newtheorem{lemma}[definition]{Lemma}
\newtheorem{prop}[definition]{Proposition}
\newtheorem{theorem}[definition]{Theorem}
\newtheorem{cor}[definition]{Corollary}

\newtheorem{remark}[definition]{Remark}
\theoremstyle{definition}
\newtheorem{exam}[definition]{Example}
\newtheorem{fact}[definition]{Fact}

\newcommand*{\AR}{\mathop{{\rm AR}}\nolimits}

\newcommand*{\End}{\mathop{{\rm End}}\nolimits}

\newcommand*{\Ext}{\mathop{{\rm Ext}}\nolimits}
\newcommand*{\Hom}{\mathop{{\rm Hom}}\nolimits}

\newcommand*{\Latt}{\mathop{{\rm Latt}}\nolimits}

\newcommand*{\Mod}{\mathop{{\rm Mod}}\nolimits}

\newcommand*{\pp}{\mathop{{\rm pp}}\nolimits}

\newcommand*{\Supp}{\mathop{{\rm Supp}}\nolimits}
\newcommand*{\Tf}{\mathop{{\rm Tf}}\nolimits}

\newcommand*{\Tor}{\mathop{{\rm Tor}}\nolimits}

\newcommand{\Zg}{\mathop{{\rm Zg}}\nolimits}

\newcommand*{\bsm}{\left(\begin{smallmatrix}}
\newcommand*{\esm}{\end{smallmatrix}\right)}
\newcommand*{\bp}{\begin{pmatrix}}
\newcommand*{\ep}{\end{pmatrix}}

\newcommand*{\xr}{\xrightarrow}

\newcommand*{\C}{\mathbb{C}}
\newcommand*{\N}{\mathbb{N}}
\newcommand*{\Q}{\mathbb{Q}}

\newcommand*{\Z}{\mathbb{Z}}

\renewcommand*{\phi}{\varphi}

\date{}

\begin{document}

\footskip=30pt

\title[]{The torsionfree part of the Ziegler spectrum of orders over Dedekind domains}

\author[]{Lorna Gregory}

\address[L.~Gregory]{University of Camerino, School of Science and Technologies,
Division of Mathematics, Via Madonna delle Carceri 9, 62032 Camerino, Italy}

\email{lorna.gregory@gmail.com}

\author[]{Sonia L'Innocente}

\address[S.~L'Innocente]{University of Camerino, School of Science and Technologies,
Division of Mathematics, Via Madonna delle Carceri 9, 62032 Camerino, Italy}

\email{sonia.linnocente@unicam.it}

\author[]{Carlo Toffalori}

\address[C.~Toffalori]{University of Camerino, School of Science and Technologies,
Division of Mathematics, Via Madonna delle Carceri 9, 62032 Camerino, Italy}

\email{carlo.toffalori@unicam.it}

\thanks{The second and third authors thank the Italian GNSAGA-INdAM for its support.}



\begin{abstract}
We study the $R$-torsionfree part of the Ziegler spectrum of an order $\Lambda$ over a
Dedekind domain $R$. We underline and comment on the role of lattices over
$\Lambda$. We describe the torsionfree part of the spectrum when $\Lambda$
is of finite lattice representation type.
\end{abstract}

\maketitle

\pagestyle{plain}

\vspace{3mm}

\section{Introduction}\label{S-intro}

In his posthumous paper \cite{Pu} Gena Puninski made substantial progress in the description
of the Cohen-Macaulay part of the Ziegler spectrum over Cohen-Macaulay rings.

Gena also raised a similar question for torsionfree modules over orders. In fact this topic
has been investigated in just few cases.
The most advanced (see \cite{PT11}) deals with the integral group ring $\mathbb{Z} G$, where
$G$ is Klein's four group; note that $\mathbb{Z} G$ is of infinite lattice
representation type. On the other hand a general theoretical analysis over group rings
$RG$, where $R$ is a Dedekind domain of characteristic 0 and $G$ is a finite group, had
been previously developed in \cite{MPT}. Motivations came from the interest in the model theory of abelian-by-finite groups. Recall that a group $H$ is said to be abelian-by-finite if and only if $H$ admits
an abelian normal subgroup $N$ of finite index. Let $G$ denote the quotient group $H/N$. So $N$ inherits a natural structure of module over the group ring $\Z G$ and most model theory of $H$ as a group is given by that of $N$ as a module (see \cite{MPT0}).

In this note we plan to extend the approach of \cite{MPT} to orders over Dedekind domains.
Thus let us first introduce this setting. We start with a Dedekind domain $R$ that is not a field.
Let $Q$ be the field of fractions of $R$. Now let $A$ be a finite dimensional $Q$-algebra. We will sometimes assume $A$ {\sl semisimple},
and even {\sl separable} (i.e. $A$ remains semisimple when extending scalars).
An {\sl $R$-order} in $A$ is a subring $\Lambda$ of $A$ such that the centre of $\Lambda$
contains $R$, $\Lambda$ is finitely generated as an $R$-module and $Q \cdot \Lambda = A$.
For instance the group ring  $RG$, with $G$ a finite group, is an $R$-order in the
group algebra $A = Q G$.

We are interested in (say right) modules over such an order $\Lambda$, in particular
in $R$-torsionfree $\Lambda$-modules. Recall that
a $\Lambda$-module $M$ is {\sl $R$-torsionfree} if for all $0 \neq m \in M$ and $0 \neq r \in R$
we have $mr \neq 0$. Finitely generated $R$-torsionfree $\Lambda$-modules
are known as {\sl $\Lambda$-lattices}. Over a Dedekind domain $R$, $\Lambda$-lattices can
be equivalently introduced as $\Lambda$-modules finitely generated and projective over $R$,
or also as direct summands (still over $R$) of $R^n$ for some positive integer $n$.

Observe that, if $M$ is any $\Lambda$-module, then $\Tor  M \, : = \, \{ m \in M \, \mid
\, mr = 0$ for some $r \in R \setminus \{ 0 \} \}$ is a $\Lambda$-submodule of $M$
and the quotient module $M / \Tor  M$ is $R$-torsionfree.

Let us fix some further notation. For $\Lambda$ an order over a Dedekind domain $R$,
\begin{itemize}
\item $\Tf_\Lambda$ is the category of all $R$-torsionfree (right) $\Lambda$-modules,
%\item ${_\Lambda}\Tf$ is the category of all $R$-torsionfree left $\Lambda$-modules,
\item $\Latt_\Lambda$ is the category of (right) $\Lambda$-lattices.
%\item ${_\Lambda}\Latt$ is the category of left $\Lambda$-lattices.
\end{itemize}

Moreover $L_\Lambda$ is the first order language of $\Lambda$-modules, and
in this language $T^{tf}_\Lambda$ is the first order theory of $R$-torsionfree
$\Lambda$-modules, so of $\Tf_\Lambda$.

Note that $\Tf_\Lambda$ is the smallest {\sl definable} subcategory of the
category of all $\Lambda$-modules containing $\Latt_\Lambda$.

For every positive integer $n$, $\pp_\Lambda^n$ denotes the lattice of pp-formulas
with $n$ free variables of $L_\Lambda$ (warning: here the word {\sl lattice} has a different
meaning, that of an ordered structure, see below). In detail, $\pp_\Lambda^n$
is the quotient set of these pp-formulas with respect to the logical equivalence
relation (in the theory of $\Lambda$-modules). The lattice structure is given by the
partial order relation determined by logical implication (modulo the same theory).
Then meet corresponds to the conjunction of pp-formulas,
and join to their sum $+$. If one identifies pp-formulas in $n$ variables equivalent
in the first order theory of some given $\Lambda$-module $M$, then one
forms another lattice $\pp^n_\Lambda (M)$ -- a quotient lattice of $\pp^n_\Lambda$.
The same can be done starting from a class of $\Lambda$-modules instead of
a single $M$. For instance,
one builds in this way $\pp^n_\Lambda (\Tf_\Lambda)$.
%-- from now on, $\pp^n_\Lambda (\Tf)$ for short.

We will denote the binary relation in these lattices by $\leq$ (with the usual meaning
for $<$). When necessary,
a subscript will specify to which lattice we refer. For instance, we write
$\leq_{\pp_\Lambda^1 (M)}$ when dealing with pp-formulas in 1 free variable
with respect to the first order theory of a module $M$. Likewise $[ \, \, , \, \, ]$
will denote a closed interval in a lattice, with possible use of subscripts to say
which lattice we deal with, as before.
Similar conventions will regard open or half closed intervals.

The m-dimension of these lattices $L = \pp_\Lambda^n, \pp_\Lambda^n (\Tf_\Lambda), \ldots$
is defined as follows, see \cite{Preb1} and \cite{Preb2} for details. Construct a sequence of
lattices $L^\alpha$ (with $\alpha$ an ordinal) collapsing at each successor step intervals of
finite length. For instance, in the basic step two pp-formulas $\phi$ and $\psi$ are identified if
and only if the closed interval $[\phi \land \psi, \phi + \psi]$ is of finite length. Then the
m-dimension of $L$, $\text{m-dim} \, L$, is defined as the smallest ordinal $\alpha$ such
that $L^\alpha$ is the one-point lattice, if such an ordinal exists, and $\infty$ (or undefined) otherwise.

Let us come back to illustrate the aim of the paper. As said, we consider $R$-torsionfree modules
over
an order $\Lambda$ over a Dedekind domain $R$. Let $\Zg_\Lambda$ denote the whole
(right) Ziegler spectrum of $\Lambda$, that is, the topological space of (isomorphism classes
of) indecomposable pure injective $\Lambda$-modules. A basis of open sets of the
topology is given by $(\phi / \psi) = \{ N \in \Zg_\Lambda \, : \, \phi \land \psi
<_{\pp^1_\Lambda (N)} \phi \}$ where $\phi$ and $\psi$ range over $\pp^1_\Lambda$.
We are interested in the subset of $\Zg_\Lambda$ formed by
$R$-torsionfree indecomposable pure injective $\Lambda$-modules.
Notice that this is a closed set, as the complement of the union of $(xr = 0 / x = 0)$
where $r$ ranges over the non zero elements of $R$. Let $\Zg_\Lambda^{tf}$ denote it.
Observe that since $\Zg_\Lambda$ is compact and  $\Zg_\Lambda^{tf}$ is a closed subset of $\Zg_\Lambda$, $\Zg_\Lambda^{tf}$ is also compact.

One may wonder what is the role of $\Lambda$-lattices in this framework,
for instance if indecomposable $\Lambda$-lattices are pure injective, so
points of $\Zg_\Lambda^{tf}$. We cannot expect that in general, but we will see
in the next section that the answer is
positive at least over complete discrete valuation domains. Apart from this, we will
also discuss the relevance of lattices in the $R$-torsionfree part of the spectrum, just
as \cite{MPT} did over group rings.

Here is a more detailed plan of this paper. In $\S$ \ref{S-latt} we prove some first results
on lattices, and above all that, when $R$ is a complete discrete valuation domain and $A$ is separable,
they are isolated points of $\Zg_\Lambda^{tf}$, dense in the whole space
$\Zg_\Lambda^{tf}$. In $\S$ \ref{S-zieg} we provide a description
of the torsionfree part of the Ziegler spectrum of an order $\Lambda$ over
a Dedekind domain $R$ in a semisimple $Q$-algebra, extending that over
group rings in \cite{MPT}. We also investigate the m-dimension of $\pp_\Lambda^1(\Tf_\Lambda)$ in that section.
%$\S$ \ref{S-maranda} is devoted to a generalization of Maranda's
%theorem from the original setting of lattices to suitable torsionfree
%pure injective modules over
%an order $\Lambda$ over a discrete valuation domain in a separable $Q$-algebra.
Applications of the (classical) Maranda theorem to our setting will be treated in
$\S$ 	\ref{S-marapp}. The final section \ref{S-finite} considers
orders of finite lattice representation type and provides a complete
description of their Ziegler spectrum, confirming a conjecture of Gena
Puninski. As an application, it is shown
that the theory of integral group rings $\Z \, G$ torsionfree over
$\Z$ (with $G$ a cyclic group of order $p$ or $p^2$ for some prime $p$)
is decidable, which positively answers questions in \cite{MPT0}.

We assume some familiarity with model theory of modules, as treated in
\cite{Preb1}, \cite{Preb2} and \cite{Zi}.

Finally let us call again the reader's attention to the fact that, as this introduction
already witnesses, the word {\sl lattice} denotes in this paper
two different concepts: lattice as a module, and as a partially ordered set.
Indeed the same is true of {\sl order}, that can be meant in the usual sense
but also as a ring.
We hope this coincidence will not cause any misunderstanding and
the meaning of any occurrence of {\sl lattice} or {\sl order} will always be clear.

\section{The role of lattices}\label{S-latt}

We mainly devote this section to some first results on lattices. We keep $R$,
$\Lambda$, and so on, in agreement with the introduction.

Since $R$ is hereditary and noetherian, $\Lambda$-lattices are closed under
submodules. On the contrary, quotients of lattices need not be lattices, but
the category of $\Lambda$-lattices does have pseudo-kernels.

%If $f \, : \, A \rightarrow B$ is a homomorphism between $\Lambda$-lattices,
%$g \, : \, B \rightarrow \coker(f)$ is a cokernel for $f$ and $\overline{g} \, : \,
%\coker (f) \rightarrow \coker (f) / \Tor \coker (f)$ is the composition of $g$
%with the canonical projection of $\coker (f)$ onto $\coker (f) / \Tor \coker (f)$,
%then all homomorphisms between $\Lambda$-lattices $h \, : \, B \rightarrow C$
%such that $hf = 0$ factor through $\overline{g}$. Thus $\overline{g}$ is a
%pseudo-kernel.

\begin{fact}\label{oplus}
{\sl Every lattice $M$ over an $R$-order $\Lambda$ decomposes as a finite direct
sum of indecomposable lattices.}
\end{fact}

\begin{proof}
Let $M = L \oplus N$. As tensor products preserve direct sums,
$QM$ decomposes as $QL \oplus QN$. Thus if both $L$ and $N$ are non zero
then $\dim QL, \, \dim QN < \dim QM$.
\end{proof}

This decomposition may not be unique. In other words the category of lattices
over an order may not be Krull-Schmidt (see \cite[p. 768]{CR}. But this is true
over complete discrete valuation domains (see \cite[(30.6), p. 620]{CR}).

\begin{prop}\label{lattpi}
Let $R$ be a complete discrete valuation domain and $\Lambda$ be an $R$-order.
Then every $\Lambda$-lattice $L$ is pure injective and the set of indecomposable $\Lambda$-lattices is dense in $\Zg_\Lambda^{tf}$.
\end{prop}

\begin{proof} 
%Let $\pi$ denote a generator of the Jacobson radical of $R$. Because each $\Lambda$-lattice is complete in the topology determined by the powers $\pi^h$, $h$ a natural number, it follows that it is algebraically compact, hence, if indecomposable, a point in $\Zg_\Lambda^{tf}$.

Since $R$ is a discrete valuation domain and $L$ is finitely generated and torsionfree over $R$, as an $R$-module, $L$ is isomorphic to $R^n$. Since $R$ is complete, it is linearly compact as a module over itself, see \cite[Sect. 4.2.2.]{Preb2} for the definition of linear compactness. Since the class of linearly compact modules is closed under extensions \cite[4.2.10]{Preb2}, $R^n$ is linearly compact. Since all pp-definable subgroups of $L$ as a $\Lambda$-module are $R$-submodules, $L$ is algebraically compact over $\Lambda$.  

%Secondly, since $R$ is complete, the category of lattices has almost split sequences
%(see \cite{RS76}, \cite{Rog95}, \cite{AS82}). Now standard arguments (see below for
%more details) show that every indecomposable $\Lambda$-lattice is isolated.

If $M$ is an $R$-torsionfree $\Lambda$-module, then it is a direct limit of its finitely
generated submodules, which are lattices. Then $M$ is in the closure of these lattices.

\end{proof}

When $Q\Lambda$ is separable and $R$ is complete, the category of lattices has almost split sequences
(see \cite{RS76}, \cite{Rog95}, \cite{AS82}). We will use this to show that every indecomposable $\Lambda$-lattice is isolated.

The following result may have its own interest and indeed will be used also later.

Let $\phi \in \pp_{\Lambda}^n$ and $(M, \mathbf{m})$, $\mathbf{m} \in M^n$,
a free realisation of $\phi$ \cite[1.2.2 p. 23]{Preb2}.
Look at the pp-type of $\mathbf{m}+\Tor M$ in $M / \Tor M$ and take a pp-formula
$\overline{\phi} \in \pp_{\Lambda}^n$ generating this pp-type.

\begin{lemma}\label{overline}
The map $\phi \mapsto \overline{\phi}$ defines a $+$-semi-lattice
homomorphism from $\pp_{\Lambda}^n$ to $\pp_{\Lambda}^n$ such that $\overline{\phi} \leq \phi$
for every $\phi$.

Moreover, if $N$ is an $R$-torsionfree $\Lambda$-module then $\phi(N)=\overline{\phi}(N)$.

The partially ordered set $\{\overline{\phi} \, \mid \, \phi\in \pp_{\Lambda}^n\}$ is isomorphic to $\pp_\Lambda^n(\Tf_\Lambda)$ and hence is a lattice.
\end{lemma}

\begin{proof}
Let $(M, \mathbf{m})$ be a free realisation of
$\phi \in \pp_{\Lambda}^n$. Since there is a homomorphism from $M$ to $M / \Tor M$
sending $\mathbf{m}$ to $\mathbf{m}+\Tor M$, $\overline{\phi}\leq \phi$.

Suppose that $\psi \leq \phi$ in $\pp_{\Lambda}^n$. Let $(N, \mathbf{n})$
be a free realisation of $\psi$. Since $\psi\leq \phi$ there is a homomorphism
$f : M \rightarrow N$ with $f(\mathbf{m}) = \mathbf{n}$. Let $\overline{f} : M / \Tor M
\rightarrow N / \Tor N$ be the homomorphism induced by $f$. Then $\overline{f}
(\mathbf{m} + \Tor M) = \mathbf{n} + \Tor N$.
Thus $\overline{\psi} \leq \overline{\phi}$. This also shows that the map sending $\phi$
to $\overline{\phi}$ is well-defined.

We now just have to observe that for all $\phi,\psi \in \pp_\Lambda^n$, $\overline{\phi+\psi}=\overline{\phi}+\overline{\psi}$. This is true because
if $(M, \mathbf{m})$, $(N, \mathbf{n})$ are free realisations of $\phi$, $\psi$
respectively, then, see \cite[1.2.27]{Preb2}, $(M \oplus N, \mathbf{m} + \mathbf{n})$ is a free realisation
of $\phi+\psi$ and $\Tor (M \oplus N) = \Tor M \oplus \Tor N$.

Now suppose that $\phi \in \pp_\Lambda^n$ is freely realised by $(M, \mathbf{m})$
and $N$ is an $R$-torsionfree $\Lambda$-module. Suppose $\mathbf{n}\in \phi(N)$.
There exists $f : M \rightarrow N$ such that $f(\mathbf{m})=\mathbf{n}$. Since $N$
is $R$-torsionfree, $\Tor M\subseteq \ker  f$. Thus the homomorphism $\overline{f} :
M / \Tor M \rightarrow N$ induced by $f$ satisfies $\overline{f}(\mathbf{m} + \Tor M) =
\mathbf{n}$. Hence $\mathbf{n}\in \overline{\phi}(N)$.

\end{proof}

We now provide a detailed proof that the indecomposable $\Lambda$-lattices are isolated in $\Zg_{\Lambda}^{tf}$, when $R$ is complete and $Q\Lambda$ is separable, following that of the analogous result for Artin algebras.

\begin{lemma}\label{lattiso}
Let $R$ be a complete discrete valuation domain with field of fractions $Q$,
$\Lambda$ be an $R$-order in a finite dimensional separable $Q$-algebra $A$.
Then all indecomposable $\Lambda$-lattices are isolated in $\Zg_\Lambda^{tf}$.
\end{lemma}
\begin{proof}
As said, the category of $\Lambda$-lattices has left almost split morphisms (see \cite[2.1]{AS82}, for instance). That is, for all indecomposable lattices $N$ there exists a homomorphism of lattices
$ f : N \rightarrow
E$ such that $f$ is a non split monomorphism and for any $\Lambda$-homomorphism of lattices $h : N \rightarrow X$ which is a non split monomorphism, there exists $\lambda \in \Hom(E,X)$ such that $f \lambda=h$.

%sequences according to \cite{RogSchARS} [REF:Auslander]. That is, for every non-injective %indecomposable $\Lambda$-lattice $N$ there exists an almost split sequence
%\[\xymatrix@C=0.5cm{
%  0 \ar[r] & N \ar[rr]^{f} && E \ar[rr]^{g} && M \ar[r] & 0 }.\] That is, the above sequence is

Pick $\mathbf{n}\in N$ a generating tuple for $N$. Let $\phi$ generate the pp-type of
$\mathbf{n}$ and $\psi$ generate the pp-type of $f(\mathbf{n})$. We first show that $N \in \left( \phi / \psi \right)$. By definition $\mathbf{n}\in\phi(N)$. Suppose, for a contradiction, that
$\mathbf{n}\in\psi(N)$. Then there exists $g:E\rightarrow N$ sending $f(\mathbf{n})$ to $\mathbf{n}$.
Since $\mathbf{n}$ is a generating tuple for $N$, $g f=\text{Id}_{N}$. But this contradicts our assumption that $f$ is not split. Thus $N\in\left(\phi/\psi\right)$.

Now we take any $\sigma \in \pp_\Lambda^n (\Tf_\Lambda)$ and we claim that, if $\sigma < \phi$ then
$\sigma \leq \psi$. Since we are working modulo the theory of $R$-torsionfree $\Lambda$-modules, we may replace $\sigma$ by $\overline{\sigma}$ (see Lemma \ref{overline}). Let $M\in \Latt_\Lambda$ and $\mathbf{m}\in M$ be such that $(M, \mathbf{m})$ is a free realisation of $\overline{\sigma}$. Thus there is a homomorphism $h : N \rightarrow M$ such that $h(\mathbf{n}) = \mathbf{m}$. So either $h$ is a split monomorphism or there exists $\lambda\in \Hom(E,X)$ such that $f\lambda=h$. If $h$ is a split
monomorphism then the pp-type of $\mathbf{n}$ is equal to the pp-type of $f(\mathbf{n}) =
\mathbf{m}$, so $\mathbf{\sigma}=\phi$. In the second case, $\mathbf{\sigma} \leq \psi$. Thus $\sigma\leq_{\Tf_{\Lambda}} \psi$.

Therefore $\phi/\psi$ is a $\Tf_\Lambda$-minimal pair. Hence $(\phi/\psi)$ isolates $N$ in $\Zg_\Lambda^{tf}$.
\end{proof}

Next let us deal with the closed points of $\Zg^{tf}_\Lambda$. We start by proving an auxiliary result.

\begin{lemma}\label{ppdefQN}
For all $N\in\Tf_\Lambda$ and $\phi\in\pp_\Lambda^1$, $Q\phi(N)=\phi(QN)$.
\end{lemma}
\begin{proof}
Let $\phi(x) \doteq \exists \mathbf{y} \ (x, \, \mathbf{y})T_\phi=0$ where $T_\phi$ is a matrix of a suitable size with entries in $\Lambda$. Since $N$ is a submodule of $QN$, $\phi(N)\subseteq \phi(QN)$. All pp-definable subsets of $QN$ are $Q$-vector subspaces, so $Q\phi(N)\subseteq \phi(QN)$.

Now suppose that $m\in \phi(QN)$. There exists $\mathbf{l}=(l_1,\ldots,l_k)$ in $QN$ such that $(m,\mathbf{l})T_\phi=0$. Let $c\in R$ be such that $mc,l_1c,\ldots l_kc\in N$. Then $(mc,\mathbf{l}c)T_{\phi}=0$. So $mc\in\phi(N)$. Thus $m\in Q\phi(N)$. Therefore $Q\phi(N)=\phi(QN)$.
\end{proof}

When $R$ is a complete noetherian valuation domain, we are now able to describe the closure of a $\Lambda$-lattice.

\begin{prop}\label{closure}
Let $R$ be a complete discrete valuation domain and $\Lambda$ an order over $R$. Let $\pi$ denote a generator of the maximal ideal of $R$.
If $N$ is an indecomposable $\Lambda$-lattice and $M$ is in the (Ziegler) closure of
$N$ but is not equal to $N$, then $M$ is a direct summand of $QN$. In particular $M$ is a closed point and $\pp_\Lambda^1(M)$ is of finite length.
\end{prop}
\begin{proof}
Let $M$ be in the closure of $N$ but not equal to $N$. Suppose that $M\in\left(\phi/\psi\right)$. We aim to show that $\phi(QN)\supsetneq \psi(QN)$.

Since $M\in\left(\phi/\psi\right)$ and $M$ is in the closure of $N$, $N\in\left(\phi/\psi\right)$. If $\phi(N)/\psi(N)$ were finite length as an $R$-module then the interval $[\psi,\phi]_N\subseteq \pp_\Lambda^1(N)$ would be finite length. Let $\phi=:\phi_0\geq \phi_1\geq\ldots\geq\phi_{n+1}:=\psi$ be such that $[\phi_{i+1},\phi_i]_N$ is simple. Then $\left(\phi/\psi\right)=\bigcup^n_{i=0}\left(\phi_i/\phi_{i+1}\right)$. Since $[\phi_{i+1},\phi_i]_N$ is simple, by \cite[7.10]{Zi}, $\left(\phi_i/\phi_{i+1}\right)$ isolates $N$ in its closure and hence $\left(\phi/\psi\right)$ isolates $N$ in its closure. Therefore $\phi(N)/\psi(N)$ is infinite length as an $R$-module.

Since $R$ is noetherian and $N$ is finitely generated as an $R$-module, $\phi(N)/\psi(N)$ is finitely generated as an $R$-module. Thus $\phi(N)/\psi(N)$ is isomorphic to $R^n\oplus T$ where $T$ is a finitely generated torsion $R$-module. Since $\phi(N)/\psi(N)$ is infinite length as an $R$-module, $n\geq 1$. Thus there exists $m\in \phi(N)$ such that $m\pi^l\notin \psi(N)$ for all $l\in\N$. By Lemma \ref{ppdefQN}, $m\in\phi(QN)$ and $m\notin\psi(QN)$.

Since $N$ is a lattice, $QN$ is finite dimensional. Let $L_1,\ldots,L_m$ be the indecomposable summands of $QN$. If $\phi(QN)\supsetneq\psi(QN)$ then $L_i\in \left(\phi/\psi\right)$ for some $1\leq i\leq m$. Thus if $M\in\left(\phi/\psi\right)$ then $L_i\in \left(\phi/\psi\right)$ for some $1\leq i\leq m$. So $M$ is in the closure of $\{L_1,\ldots,L_m\}$. Since each $L_i$ is a closed point, $M$ is in the closure of $\{L_1,\ldots,L_m\}$ if and only if $M=L_i$ for some $1\leq i\leq m$.
\end{proof}

Recall that the {\sl support} of a $\Lambda$-module $M$, $\Supp (M)$, is the
set of indecomposable pure injective $\Lambda$-modules $N$ such that for all pp-pairs
$\phi/\psi$, $\phi(M)=\psi(M)$ implies $\phi(N)=\psi(N)$. In particular,
if $M \in \Zg_\Lambda$ then $\Supp (M)$ is the closure of $M$ in $\Zg_\Lambda$.

\begin{lemma}\label{close1}
If $N \in \Zg_\Lambda^{tf}$  and $S \in \Supp (QN)$ then $S$ is in the closure of $N$.
In particular, if $N\in \Zg_\Lambda^{tf}$  and $S\in\Zg_A$ is a direct summand of $QN$,
then $S$ is in the closure of $N$.
\end{lemma}

\begin{proof}
Suppose that $\phi\geq\psi\in\pp_\Lambda^1$. By \ref{ppdefQN}, if $\phi(N)=\psi(N)$ then $\phi(QN)=Q\phi(N)=Q\psi(N)=\psi(QN)$. By definition, if $S\in\Supp (QN)$ then $\phi(QN)=\psi(QN)$ implies $\phi(S)=\psi(S)$. Thus $S$ is in the closure of $N$.
\end{proof}

\begin{cor}\label{closedina}
If $N\in \Zg_\Lambda^{tf}$ is a closed point then $N\in\Zg_A$ and $N$ is a closed point in $\Zg_A$.
\end{cor}

We include in this section some further useful remarks.
For every $R$-module $M$, let $\text{ Sub}_R (M)$ be the lattice of $R$-submodules of $M$.

\begin{lemma}\label{md1}
Let $R$ be a Dedekind domain and $M$ a torsionfree finitely generated module over $R$. Then
$\text{ Sub}_R (M)$ has m-dimension $1$.
\end{lemma}
\begin{proof}
Since $M$ is torsionfree, $R$ is a submodule of $M$ and hence $\text{ Sub}_R (M)$ is not of finite length.
Thus $\text{m-dim} \, \text{ Sub}_R (M) \neq 0$.

Since $M$ is finitely generated and torsionfree, $M$ is a direct summand of $R^n$ for some positive integer
$n$. It follows that, if $\text{m-dim} \, \text{ Sub}_R (R^n) \leq 1$ for all $n\in\N$, then $\text{m-dim}
\, \text{ Sub}_R (M) \leq 1$ and hence $\text{m-dim} \, \text{ Sub}_RM=1$.

On the other hand, since $R^n$ can be filtered as a finite chain of submodules with quotients isomorphic to $R$, $\text{m-dim} \, \text{ Sub}_R (R^n) = \text{m-dim} \, \text{ Sub}_R (R)$.

Let $J$ be a non zero ideal of $R$. Then $R/J$ is of finite length. Thus the interval $[J,R]$ in $\text{Sub}_R
(R)$ is of finite length. So $\text{m-dim} \, \text{ Sub}_R (R) =1$.
\end{proof}

\begin{cor}\label{md2}
Let $R$ be a Dedekind domain, $\Lambda$ an order over $R$ and $M$ a $\Lambda$-lattice. Then $\pp_\Lambda^1 (M)$ has m-dimension $1$.
\end{cor}
\begin{proof}
Since all pp-definable subgroups of $M$ are $R$-submodules, $\pp_{\Lambda}^1(M)$
is a sublattice of $\text{Sub}_R (M)$. Thus m-dim $\pp_{\Lambda}^1 (M) \leq \text{ m-dim}  \text{Sub}_R (M) =1$. Since $M$ is $R$-torsionfree and not $R$-divisible, there exist $c \in R\backslash\{0\}$ and $m\in M$ such that $m\notin Mc$.
Since $M$ is $R$-torsionfree, $mc^n\notin Mc^{n+1}$. Thus $Mc^{n+1}\subsetneq Mc^n$ for all $n\in\N$. Thus
$pp_\Lambda^1 (M)$ is not of finite length, whence its m-dimension cannot be 0.
\end{proof}

\section{The torsionfree part of the Ziegler spectrum}\label{S-zieg}

In this section we extend the main results of \cite{MPT}, about the torsionfree part
of the Ziegler spectrum of a group ring $RG$, with $R$ a Dedekind domain of characteristic 0
and $G$ a finite group, to arbitrary orders $\Lambda$ over a Dedekind domain
$R$ in a {\sl semisimple} $Q$-algebra $A=Q\Lambda$. We also investigate the m-dimension of $\pp_\Lambda^1(\Tf_\Lambda)$ and the Cantor-Bendixson rank of $\Zg_\Lambda^{tf}$ in both this
framework and the more general setting where $A$ is not assumed to be semisimple.

If $P$ is a prime ideal of $R$ then we write $\Lambda_P$ for the central localisation of $\Lambda$ at $P$ and $\widehat{\Lambda_P}$ for its completion at $P$. Note that $\Lambda_P$ is an $R_P$-order in $A$ and $\widehat{\Lambda_P}$ is an $\widehat{R_P}$-order.

We start in the general setting.  The endomorphism ring of every indecomposable pure injective module $N$, $\End(N)$, is local. Let $P(N)$ denote its maximal ideal. When $N$ is $R$-torsionfree, $R$
embeds in a natural way into $\End(N)$. Moreover  $P(N)\cap R$ is a prime ideal of $R$. Thus, every indecomposable pure injective $\Lambda$-module is a module over $\Lambda_P$ for some prime
and even maximal ideal $P$ of $R$. The homomorphism $\Lambda\rightarrow \Lambda_P$ is an epimorphism and hence restriction of scalars induces an embedding of $\Zg_{\Lambda_P}$ into $\Zg_\Lambda$ whose image is a closed subset. This embedding restricts to an embedding of $\Zg_{\Lambda_P}^{tf}$ into $\Zg_{\Lambda}^{tf}$ and again the image is a closed subset. Identifying $\Zg_{\Lambda_P}^{tf}$ with the set of $N\in \Zg_{\Lambda}^{tf}$ such that $P(N)\cap R\subseteq P$, we may write \[\Zg_\Lambda^{tf}=\bigcup_{P}\Zg_{\Lambda_P}^{tf}\] where $P$ ranges over maximal ideals of $R$. Since $R$ has Krull dimension $1$, if $P,P'$ are distinct maximal ideals of $R$ then $\Zg_{\Lambda_P}^{tf}\cap\Zg_{\Lambda_{P'}}^{tf}=\Zg_A$.

This description of the space is not particularly useful for computing the Cantor-Bendixson rank of $\Zg_\Lambda^{tf}$ because if $T$ is a topological space, $X$ is a closed subset of $T$ and $p\in X$ then the Cantor-Bendixson rank of $p$ as a point in $X$ may strictly less than the Cantor-Bendixson rank of $p$ as a point in $T$. Thus we now work to give a more useful description.

Since $R$ is noetherian, every (maximal) ideal $P$ of $R$ is finitely generated,
whence there is a pp-formula of $L_R$ defining in any $R$-module $M$ just $MP$: if $\mathbf{r} =
(r_1, \ldots, r_l)$ is a generating tuple of $P$, it suffices to take $\exists y_1 \, \ldots \, \exists y_l \,
(x = y_1 r_1 + \ldots + y_l r_l )$. Let $ P \mid x$ denote this formula. For instance, when $R$ is a
discrete valuation domain and $\pi$ is a generator of its maximal ideal $P$, then the formula is
$\pi \mid x$, that is, $\exists y (x = y \pi)$.

If $N\in \Zg_{\Lambda_{P'}}^{tf}$ and $P\neq P'$ then $NP=N$ since some element of $P$ is not in $P'$, that is, some element of $P$ acts invertibly on $N$.

Now suppose that $N\in \Zg_{\Lambda_{P}}^{tf}$ and $N\notin\left(x=x/P|x\right)$. Let $(r_1, \ldots, r_l)$ still denote a tuple of generators of $P$. Since $R_P$ is a valuation domain, there exists $1\leq j \leq l$ such that $r_i \in r_j R_P$  for all $1\leq i \leq l$. Put $r = r_j$ and for all $i$ write $r_i =  r \, c_i/a_i$ with $c_i, a_i \in R$ and $a_i \notin P$. Then $r c_i = r_i a_i$ for all $i$. Set $a = \prod_{1 \leq h \leq l} a_h$. Then $a \in R \setminus P$. Multiply the $i$-th equation above by $\prod_{h \neq i} a_h$ and get for every $i$ a new equation $ r b_i = r_i a$ for a suitable $b_i \in R$.
It follows that $NP=Nr$. Hence the fact that $NP=N$ implies that $r$ acts invertibly on $N$. So $P(N) \cap
R \subsetneq P$. Therefore $P(N) \cap R$, as a prime ideal of $R$, coincides with $0$. So $N$ is $R$-divisible i.e. $N\in\Zg_A$. Thus we have shown that

\[\Zg_\Lambda^{tf} = \bigcup_P (\left( x=x / P \mid x\right) \cap \Zg_\Lambda^{tf}) \cup \Zg_A\]

\noindent
where $P$ ranges over maximal ideals of $R$.

As promised, we now generalise the main results \cite{MPT} to orders in semisimple algebras. A large part of the proof is the same as over group rings, but adaptions are sometimes necessary.

\begin{theorem}\label{spectrum}
Let $R$ be a Dedekind domain with field of fractions $Q$, and $\Lambda$
an $R$-order in a semisimple $Q$-algebra $A$. If $N \in \Zg_\Lambda^{tf}$, then either
\begin{itemize}
\item $N$ is a simple $A$-module, or
\item there is some maximal ideal $P$ of $R$ such that $N \in \Zg^{tf}_{\widehat{\Lambda_P}}$ and $N$ is $R_P$-reduced.
\end{itemize}
Moreover, if $N\in\Zg^{tf}_{\widehat{\Lambda_P}}$ is $R_P$-reduced then $N\in \Zg_\Lambda^{tf}$.
\end{theorem}

Here {\sl $N$ being $R_P$-reduced} means that $\cap_{i=0}^{\infty} N P^i = 0$. Recall that $\widehat{\Lambda_P}$ is an order over $\widehat{R_P}$ in $\widehat{A}=\widehat{Q}\otimes_QA$ where $\widehat{Q}$ denotes the field of fractions of $\widehat{R_P}$.

\begin{proof}
We follow the proof of \cite[Theorem 2.1, pp. 1127-1130]{MPT}. For simplicity
we divide our argument in several steps. Let $N$ be an indecomposable pure injective $R$-module.

\smallskip \noindent

{\sl Step 1.} For some maximal ideal $P$ of $R$, $N$ is a module over $\Lambda_P$.

This step has already been covered in the discussion preceding this theorem.
Let $\pi$ denote a generator of the maximal ideal $P R_P$ of $R_P$.

\smallskip

{\sl Step 2.} Any $\Lambda_P$-module divisible and torsionfree over $R_P$ is
injective over $\Lambda_P$.

The proof is the same as \cite[Claim 2, p. 1128]{MPT}.

\smallskip

{\sl Step 3.} $N$, as an $R_P$-torsionfree module over $\Lambda_P$,
decomposes  over $\Lambda_P$ as $N'\oplus N''$ where $N''$ is
$R_P$-divisible (hence an $A$-module) and $N'$ is $R_P$-reduced, i.e.
$\cap_{i=1}^\infty N' P^i = \cap_{i=1}^\infty N' \pi^i = 0$.

To prove this claim, first we put \[ N'':=\{m\in N \, \mid \, \pi^n | m \text{ for all } n\in\N\}.\]

Take $n\in\N$, $m,m'\in N''$ and $r\in \Lambda$. Since $m,m'\in N''$ there exists $a,a'\in N$ such that $m=a\pi^n$ and $m'=a'\pi^n$. Thus $mr+m'=a\pi^nr+a'\pi^n=(ar+a')\pi^n$ because $\pi$ is central. Thus $N''$ is a submodule of $N$. Since $N''$ is $R_P$-divisible by definition, it is injective by Step 2
and thus a direct summand of $N$.

Let $N'$ be a complement of $N''$ in $N$. If $m\in \cap_{i=1}^\infty N'\pi^i$ then $\pi^n|m$ for all $n\in\N$ and thus $m\in N''$. So $m=0$.

This concludes Step 3.

\smallskip

As $N$ is indecomposable, either
\begin{itemize}
\item[(a)] $N$ is an $A$-module, or
\item[(b)] $N$ is $R_P$-reduced.
\end{itemize}
In the former case $N$ must be a simple $A$-module. So let us turn to (b). We assume from now
on that $N$ is $R_P$-reduced.

\smallskip

{\sl Step 4.}
Every $R_P$-reduced pure injective $\Lambda_P$-module $M$ can be equipped with a
$\widehat{\Lambda_P}$-module structure, and $M$ remains pure injective over
$ \widehat{\Lambda_P}$.

To see this, one proceeds exactly as in \cite[pp. 1128-1129]{ MPT}.
Suppose $r \in \widehat{\Lambda_P}$ and $m \in M$. For each $i \in \N$, let $r_i \in \Lambda_P$
satisfy $\pi^i | r-r_i$. For each $i \in \N$, look at the equation $x - m r_i= y_i \pi^i$. When $i$
ranges over $\N$, this set of equations is finitely solvable and so, since $M$ is pure injective,
solvable in $M$.
Let $n,n' \in M$ be such that $\pi^i | n-mr_i$ and $\pi^i | n'-mr_i$ for all $i$. Then
$n-n'\in \cap_{i=1}^\infty M \pi^i = 0$.

Define $mr$ to be the unique element $n \in M$ such that  $\pi^i | n-mr_i$ for all $i$. Note that this definition of
$mr$ does not depend on the particular choice of $r_i$ above. If $r, s \in \widehat{\Lambda_P}$
then $mr$ is the unique element $m_1$ in $M$ such that $\pi^i | m_1 - mr_i$ and $ms$ is the unique element $m_2$ in $M$ such that $\pi^i | m_2 - ms_i$. Thus $\pi^i | m_1+m_2 - m(r_i+s_i)$ for
all $i\in \N$. Hence $m_1+m_2=m(r+s)$.

That $M$ is pure injective as a $ \widehat{\Lambda_P}$-module is a consequence of
\cite[Lemma p. 1129]{MPT}.

\smallskip

Conversely, independently of the assumption that $N$ is $R_P$ reduced, every pure injective
$\widehat{\Lambda_P}$-module $N$ remains pure injective after restricting it over $\Lambda_P$.
This is simply because $\Lambda_P$ is a subring of its $P$-adic completion. For the same
reason any decomposable module over $\widehat{\Lambda_P}$ is decomposable over
$\Lambda_P$. On the other hand the following holds.

\smallskip

{\sl Step 5.} If $N$ is an $R_P$-reduced indecomposable pure injective
$\widehat{\Lambda_P}$-module, then $N$ is indecomposable as a $\Lambda$-module.

This is explained in \cite[Remark 1, p. 1130]{MPT}.

\end{proof}

The above theorem has shown that, when $A$ is semisimple, $R_P$-reduced $R$-torsionfree indecomposable pure injective modules are the same
over $\Lambda_P$ and over $\widehat{\Lambda_P}$. Moreover, the set of $R_P$-reduced $R$-torsionfree indecomposable pure injective modules are exactly those modules in the open set $\left(x=x /P\mid x\right)$. For this reason we will sometimes write $\Zg_{\widehat{\Lambda_P}}^{tf+red}$ for this open set.

Note that any module $N \in \Zg_\Lambda^{tf}$ that can be regarded as a $\Lambda_P$-module but
does not belong to $\left( x = x / P \mid x \right)$, that is, satisfies $N = PN$, is an $A$-module.
Furthermore any two different sets $\left( x=x / P \mid x\right) \cap \Zg_\Lambda^{tf}$ are disjoint from each other. Thus Theorem \ref{spectrum} asserts that $\Zg_\Lambda^{tf}$ is the disjoint union of $\Zg_A$ and of the various $\Zg_{\widehat{\Lambda_P}}^{tf+red}$ where $P$ ranges over maximal ideals of $R$.

Now let us deal with the topology. Notice that for every maximal ideal $P$
of $R$ the embedding of $\Lambda$ into $\widehat{\Lambda_P}$ induces an inclusion
of the $R_P$-reduced part of $\left (x = x / P \mid x \right) \cap \Zg^{tf}_\Lambda =
\Zg_{\widehat{\Lambda_P}}^{tf+red}$ into $\Zg_\Lambda^{tf}$.
By the same argument given in \cite[Theorem 2.2]{MPT} over group rings, this inclusion
is homeomorphic:

\begin{theorem}\label{top} Suppose that $A=Q\Lambda$ and $\widehat{A}=\widehat{Q}\Lambda$ (where $\widehat{Q}$ is the field of fractions of $\widehat{R_P}$) are semisimple. For every maximal ideal $P$ of $R$,
$(x=x / P \mid x) \cap \Zg_\Lambda^{tf} = \Zg_{\widehat{\Lambda_P}}^{tf+red}$
has the same topology whether viewed as a subspace of $\Zg_{\widehat{\Lambda_P}}^{tf}$
or of $\Zg_\Lambda^{tf}$.
\end{theorem}

Here we give a different proof of a slightly stronger claim.

First of all, observe that every pp-formula $\alpha$ of $L_{\Lambda_P}$, $\alpha \doteq
\exists \mathbf{y} (\mathbf{x} S = \mathbf{y} T)$, with $S$, $T$ matrices of suitable sizes
with entries in $\Lambda_P = R_P \Lambda$, can be
translated into a pp-formula $\alpha'$ of $L_\Lambda$ equivalent to $\alpha$
in all $R$-torsionfree $\Lambda_P$-modules. To build $\alpha'$, calculate the
product $r$ of all multiplicative inverses of scalars of $R$ occurring in the
entries of $S$ and $T$. Then $r \in R \setminus P$, in particular $r \neq 0$.
Now multiply the previous scalars by $r$ and get $\alpha'$ as required, as
$\exists \mathbf{y} (\mathbf{x} rS = \mathbf{y} rT)$. In fact the entries of
$r S$ and $r T$ are in $\Lambda$. The torsionfree condition guarantees the
equivalence to $\alpha$. That is, for every $R$-torsionfree $\Lambda_P$-module
$M$ and $\mathbf{m}, \mathbf{n}$ in $M$,
$r (\mathbf{m} S - \mathbf{n} T) = 0$ if and only if $\mathbf{m} S - \mathbf{n} T = 0$.

Thus we have to compare $\Lambda_P$ and $\widehat{\Lambda_P}$. We may now assume that $R$ is a discrete valuation
domain and $\pi$ is a generator of its maximal ideal $P$; $Q$ is still the field of
fraction of $R$, $A$ a finite dimensional $Q$-algebra, $\Lambda$ an order over $R$ in
$A$, $\widehat{\Lambda}$ its $\pi$-adic completion. We also assume both
$A$ and $\widehat{A}$ semisimple, which is true, in particular when $A$ is separable.
Under these conditions we prove the following.

\begin{prop}\label{top2}
Suppose that $R$ is a discrete valuation domain whose maximal ideal is generated by $\pi$ and that both $A$ and $\widehat{A}$ are semisimple. The closed intervals $[\pi  \mid x,x=x]_{\Tf_\Lambda}$ and $[\pi \mid x,x=x]_{\Tf_{\widehat{\Lambda}}}$ are isomorphic as lattices.
Moreover the Ziegler open sets $\left(x=x/ \pi \mid x \right)$ in $\Zg^{tf}_\Lambda$ and
$\left(x=x/\pi \mid x \right)$ in $\Zg^{tf}_{\widehat{\Lambda}}$ are homeomorphic.
\end{prop}

\begin{proof}
For $1\leq j\leq n$ and $1\leq k\leq l$, let $s_j,r_{jk} \in \widehat{\Lambda}$, and let $\phi$ be
the pp-formula
\[\exists \ y_1 \exists y_2 \ldots \exists y_l \bigwedge_{j=1}^n ( xs_j+\sum_{k=1}^ly_kr_{jk}=0 ).\]
Further suppose that $\pi|x\leq \phi$.

For $i\in\N$, let $s^i_j,r^i_{jk}\in\Lambda$ be such that $\pi^i \mid s_j^i-s_j$ and
$\pi^i \mid r_{jk}^i-r_{jk}$. For each $i\in \N$, let $\phi_i$ be the pp-formula
\[\exists \ y_1 \exists y_2 \ldots \exists y_l \bigwedge_{j=1}^n \pi^i|(xs^i_j+\sum_{k=1}^ly_kr^i_{jk}).\]

Clearly, $\phi\leq \phi_i$ and $\phi_i\geq \phi_{i+1}$ for each $i\in \N$.

We now show that for all indecomposable pure injective $\widehat{\Lambda}$-modules $N$, $\bigcap_{i\in\N}\phi_i(N)=\phi(N)$.

Since $\widehat{A}$ is semisimple, every indecomposable pure injective $\widehat{R}$-torsionfree $\widehat{\Lambda}$-module $N$ is either $R_P$-reduced or $\widehat{R}$-divisible. In the
latter case, since $\phi_i\geq\phi\geq \pi|x$, $\phi_i(N)=\phi(N)=N$.
Hence assume that $N$ is reduced. Suppose that $m\in \phi_i(N)$ for all $i\in\N$. Then the infinite system of linear equations
\[ms_j^i+\sum_{k=1}^ly_kr^i_{jk}=z_j^i\pi^i\] where $i\in\N$, $1\leq j\leq n$ and $1\leq k\leq l$, is finitely solvable. Consequently, since $N$ is pure injective, it is solvable say with $y_k=a_k\in N$.
So for each $1\leq j\leq n$, $ms_j+\sum_{k=1}^la_kr^i_{jk}\in N\pi^i$ for all $i\in\N$. Thus, since $N$ is reduced, $ms_j+\sum_{k=1}^la_kr^i_{jk}=0$. Thus $m\in\phi(N)$.

We now show that there exists an $i\in\N$ such that $\phi_i=\phi$.
Suppose that $\phi<\phi_i$ for all $i\in\N$ with respect to the theory of $\widehat{R}$-torsionfree $\widehat{\Lambda}$-modules. Let $F$ be the filter generated by $\{\phi_i \, \mid \, i\in \N\}$
and $I$ be the ideal generated by $\phi$. Since $F\cap I$ is empty, by \cite[4.33]{Preb1}, we can construct an irreducible pp-type $p$ such that $\phi_i\in p$ for all $i\in \N$ and $\phi\notin p$. Now, $N(p)$, the pure injective hull of $p$, has an element $m$ such that $m\in\phi_i(N(p))$ for
all $i\in\N$ but $m\notin \phi(N(p))$. This contradicts the fact that $\phi(N)=\bigcap_{i\in\N}\phi_i(N)$ for all indecomposable pure injective $\widehat{R}$-torsionfree
$\widehat{\Lambda}$-modules $N$. Thus $\phi_i=\phi$ for some $i\in \N$.

Restriction of scalars gives a lattice homomorphism from $\pp^1_{\Lambda}(\Tf_\Lambda)$ to $\pp^1_{\widehat{\Lambda}}(\Tf_{\widehat{\Lambda}})$. We have shown that the restriction
of this map to $[\pi \mid x,x=x]_{\Tf_R(\Lambda)}$ has image $[\pi \mid x, x=x]_{\Tf_{\widehat{\Lambda}}}$. Suppose that $\phi> \psi\geq \pi|x$ in $\pp^1_{\Lambda}(\Tf_\Lambda)$.
There exists an indecomposable pure injective $\Lambda$-module $N$ such that
$\phi(N)\supsetneq
\psi(N)$. Since $N\pi\neq N$, this implies that $N$ is $R_P$-reduced and hence $N$ can
be equipped with the structure of a $\widehat{\Lambda}$-module. Hence
$\phi>\psi$ as pp-formulas in $\pp^1_{\widehat{\Lambda}}(\Tf_{\widehat{\Lambda}})$.
Thus we have shown that $[\pi \mid x,x=x]_{\Tf_\Lambda}$
and $[\pi \mid x, x=x]_{\Tf_{\widehat{\Lambda}}}$ are isomorphic as lattices.

Since both $A$ and $\widehat{A}$ are semisimple, the open sets
$\left(x=x/\pi|x\right)\subseteq \Zg^{tf}_\Lambda$ and $\left(x=x/\pi|x\right)\subseteq \Zg^{tf}_{\widehat{\Lambda}}$ contain exactly the reduced indecomposable
pure injective modules. Thus restriction of scalars gives a bijection from $\left(x=x/\pi \mid x\right) \subseteq \Zg^{tf}_{\widehat{\Lambda}}$ to $\left(x=x/\pi \mid x\right)\subseteq \Zg^{tf}_\Lambda$. Since the sets of
the form $\left(\phi/\psi\right)$ for $\pi \mid x \leq \psi < \phi$ in $\pp^1_{\Lambda}(\Tf_\Lambda)$
(respectively $\pi \mid x \leq \psi < \phi$ in $\pp^1_{\widehat{\Lambda}} (\Tf_{\widehat{\Lambda}})$) give a basis for $\left(x=x/\pi \mid x\right)\subseteq \Zg^{tf}_\Lambda$ (respectively $\left(x=x/\pi \mid x \right)\subseteq \Zg^{tf}_{\widehat{\Lambda}}$), this bijection is a homeomorphism.

\end{proof}

\begin{prop}\label{cb} Let $A$ be semisimple. Then
$\Zg_\Lambda^{tf}$ has Cantor-Bendixson rank if and only if $\Zg_A$ and each  $\Zg_{\widehat{\Lambda_P}}^{tf+red}$, with $P$ a maximal ideal of $R$, have
Cantor-Bendixson rank.
\end{prop}

The proof is similar to that at the end of the next Lemma \ref{genmdim}. So we ask the
reader to wait for that lemma, and a few lines, to see its details.

Next let us deal with the m-dimension of $\pp_\Lambda^1 (\Tf)$. Recall the connection
between m-dimension and Cantor-Bendixson rank \cite[Corollary 5.3.60]{Preb2}, at least
under the {\sl isolation condition}. The latter requires that for every closed subset $\mathcal{C}$ of
$\Zg_\Lambda$ (indeed, in our case, of $\Zg_\Lambda^{tf}$) and every isolated point $N$ of $\mathcal{C}$,
there is a pp-pair $\phi / \psi$ which is minimal such that $(\phi / \psi) \cap \mathcal{C} = \{ N \}$
(see \cite[5.3.2]{Preb2}).

We first prove a general statement which we were not able to find elsewhere in the literature.

\begin{lemma}\label{genmdim}
Let $S$ be an arbitrary ring, $X$ a closed subset of $\Zg_S$ and $\{\phi_i/\psi_i \ | \ i\in I\}$ a set of pp-pairs. Then $\pp^1_S(X)$ has m-dimension if and only if $[\psi_i,\phi_i]_X$ has m-dimension for all $i\in I$ and $\pp_S^1(X\backslash\bigcup_{i\in I}\left(\phi_i/\psi_i\right))$ has m-dimension.
\end{lemma}
\begin{proof}
If $\pp^1_S(X)$ has m-dimension then $\pp_S^1(X\backslash\bigcup_{i\in I}\left(\phi_i/\psi_i\right))$ has m-dimension since it is a quotient of $\pp^1_S(X)$ and for all $i\in I$, $[\psi_i,\phi_i]_X$ has m-dimension because $[\psi_i,\phi_i]_X$ is a sublattice of $\pp^n_S(X)$ for some $n\in\N$.

Now suppose that $[\psi_i,\phi_i]_X$ has m-dimension for all $i\in I$ and $\pp_S^1(X\backslash\bigcup_{i\in I}\left(\phi_i/\psi_i\right))$ has m-dimension. We first show that for all $N\in X$, $N$ has an $N$-minimal pair, i.e. $\pp^1_S(N)$ has a simple interval.

Firstly, suppose that $N\notin\left(\phi_i/\psi_i\right)$ for all $i\in I$. Then $\text{cl}(N)\subseteq X\backslash\bigcup_{i\in I}\left(\phi_i/\psi_i\right)$ where cl denotes closure. Thus $\pp_S^1(\text{cl}(N))=\pp_S^1(N)$ has m-dimension and hence $N$ has an $N$-minimal pair.

Now suppose that $N\in \left(\phi_i/\psi_i\right)$ for some $i\in I$. Since $[\psi_i,\phi_i]_X$ has m-dimension, so does $[\psi_i,\phi_i]_N$. Thus $[\psi_i,\phi_i]_N$ contains a simple interval. So $N$ has an $N$-minimal pair.

Therefore, by \cite[5.3.16]{Preb2}, the isolation condition holds for $X$. So, by \cite[5.3.60]{Preb2}, $X$ has Cantor-Bendixson rank if and only if $\pp^1_S(X)$ has m-dimension. We now show that all points in $X$ have Cantor-Bendixson rank.

Let $\alpha>\text{m-dim}[\psi_i,\phi_i]_X$ for all $i\in I$. Then, by \cite[5.3.59]{Preb2}, $\phi_i(M)=\psi_i(M)$ for all $M\in X^{(\alpha)}$. Therefore $X\backslash \bigcup_{i\in I}\left(\phi_i/\psi_i\right)\supseteq X^{(\alpha)}$. Since $\pp_S^1 (X \backslash \bigcup_{i\in I}\left(\phi_i/\psi_i\right))$ has m-dimension, so does $\pp_S^1(X^{(\alpha)})$. So $X^{(\alpha)}$ has Cantor-Bendixson rank and hence, $X$ also has Cantor-Bendixson rank.

\end{proof}

\begin{cor}\label{mdim}
The lattice $\pp^1_\Lambda (\Tf_\Lambda)$ has m-dimension if and only if, for all maximal ideals $P$ of $R$,
$[P \mid x,x=x]_{\Tf_\Lambda}$ has m-dimension and $\pp^1_{Q\Lambda}$ has m-dimension.
\end{cor}

\begin{proof}
Note that $\Zg_\Lambda^{tf}\backslash \bigcup_P \left(x=x/P \mid x\right)$ just consists of the $R$-divisible modules i.e. of modules over $Q\otimes \Lambda$.
\end{proof}

The following is an easy consequence of the beginning of the proof of \ref{genmdim}.

\begin{remark}
m-dim $\pp^1_\Lambda(\Tf_\Lambda) \geq$ m-dim $[P \mid x, x=x]_{\Tf_\Lambda}$ for every
maximal ideal $P$ of $R$, and
m-dim $\pp^1_\Lambda(\Tf_\Lambda) \geq$ m-dim $\pp^1_A$.
\end{remark}

Specialising to the case where $A=Q\Lambda$ is semisimple we get the following.

\begin{lemma}
Suppose that $A=Q\Lambda$ is semisimple and $[P \mid x, x=x]_{\Tf_\Lambda}$ has m-dimension for every
maximal ideal $P$ of $R$. Let $\alpha$ be the least ordinal such that $\alpha>\text{m-dim}[P|x,x=x]_{\Tf_\Lambda}$ for all maximal ideals $P$ of $R$. Then $\text{m-dim}\pp_\Lambda^1(\Tf_\Lambda)=\alpha$.
\end{lemma}
\begin{proof}
Since $A$ is semisimple, $\pp^1_A$ has m-dimension zero. So by \ref{mdim}, $\pp_\Lambda^1(\Tf_\Lambda)$ has m-dimension. Thus the m-dimension of $\pp_\Lambda^1(\Tf_\Lambda)$ is equal to the Cantor-Bendixson rank of $\Zg_\Lambda^{tf}$.

Note that $X^{(\alpha)}\cap\left(x=x/P|x\right)=\emptyset$ for all maximal ideals $P$ of $R$. Thus $\Zg_{Q\Lambda}\supseteq X^{(\alpha)}$. Since $A$ is semisimple, all points in $\Zg_A$ are isolated. Thus $X^{(\alpha+1)}=\emptyset$. Thus we just need to show that $X^{(\alpha)}\neq \emptyset$. We deal with the cases where $\alpha$ is a successor ordinal and $\alpha$ is a limit ordinal separately.

Suppose $\alpha=\beta+1$. There exists a maximal ideal $P$ of $R$ such that $\beta=\text{m-dim}[P|x,x=x]$. Thus $X^{(\beta)} \cap \left(x=x/P|x\right)\neq \emptyset$. Take $N\in X^{(\beta)}\cap \left(x=x/P|x\right)$ and let $L$ be an indecomposable direct summand of $QN$. By Lemma \ref{close1}, $L$ is in the closure of $N$. Thus $L\in X^{(\alpha)}$.

Now suppose $\alpha$ is a limit ordinal. For all $\beta<\alpha$, there exists a maximal ideal $P$ of $R$ such that $\left( x=x/P|x\right)\cap X^{(\beta)}$ is non-empty. Thus $X^{(\beta)}\neq \emptyset$ for all $\beta<\alpha$. Since $\Zg_\Lambda^{tf}$ is compact, $X^{(\alpha)}=\bigcap_{\beta<\alpha}X^{(\beta)}$ is non-empty.
\end{proof}

\begin{cor}\label{mdimloctoglob}
Suppose that $A=Q\Lambda$ is semisimple and $\pp^1_{\Lambda_P}(\Tf_{\Lambda_P})$ has m-dimension for every
maximal ideal $P$ of $R$. Then the m-dimension of $\pp^1_\Lambda(\Tf_\Lambda)$ is equal to the supremum of the m-dimensions of $\pp^1_{\Lambda_P}(\Tf_{\Lambda_P})$ where $P$ ranges over maximal ideals of $R$. Moreover, if $A=Q\Lambda$ is separable, the m-dimension of $\pp^1_\Lambda(\Tf_\Lambda)$ is equal to the supremum of the m-dimensions of $\pp^1_{\widehat{\Lambda_P}}(\Tf_{\widehat{\Lambda_P}})$ where $P$ ranges over maximal ideals of $R$.
\end{cor}

We will come back to m-dimension and Cantor-Bendixson rank after dealing
with orders of finite lattice representation type. We will show, in \ref{md3}, that if $\Lambda$ has finite lattice representation type then the m-dimension of $\pp^1_\Lambda(\Tf)$ is $1$. This will allow us, in \ref{md4}, to improve the above corollary when $Q\Lambda$ is separable.

% Before that, we focus in the next section
%on the space $\Zg^{tf+red}_{\widehat{\Lambda_P}} = (x=x / P \mid x) \cap \Zg_{\widehat{\Lambda}}^{tf}$
%for a given maximal ideal $P$ of $R$ and we see how reduction modulo $P^k$
%for $k$ a positive integer large enough and an adaption of Maranda's Theorem
%(\cite[Theorem 30,14 p. 624]{CR} can provide further information on its topology.

\section{Applications of Maranda's Theorem}\label{S-marapp}

Let us open a short parenthesis on Maranda's Theorem \cite[$\S$ 30A]{CR}. Following
\cite[30.12]{CR}, we assume throughout this section  that $R$ is a discrete valuation domain, $\pi$
is a generator of its maximal
ideal $P$, $Q$ is the field of fractions of $R$, $A$ is a finite dimensional {\sl separable}
$Q$-algebra and $\Lambda$ is an $R$-order in $A$.

We will deal with the quotient ring $\Lambda / \pi^k \Lambda$, often abbreviated as
$\Lambda_k$, for every positive integer $k$. Similarly, for every $\Lambda$-module $M$,
$M_k$ will denote the quotient module $M / \pi^k M$.
%When $k$ is clear from the context, for $M$ a
%$\Lambda$-module and $m\in M$, we will sometimes write $\overline{M}$ for $M/\pi^kM$
%and $\overline{m}$ for $m+\pi^kM$.
We will write $\mathbf{x}$ for tuples of variables and likewise, as in the previous
sections, $\mathbf{m}$ for tuples of elements in a module.

There is a non negative integer, and hence a minimal non negative integer $k_0$
such that $\pi^{k_0}\Ext^1(L, N)=0$ for all $\Lambda$-lattices $L$ and $N$ (see \cite[p. 624]{CR}).
For instance, when $\Lambda = R \, G$ for some finite group $G$, then $k_0$ is
the largest integer such that $|G| \in \pi^{k_0} R$.
Note that, since $\Lambda$ is noetherian, $\Ext^1(L,-)$ is a finitely presented functor from
$\Lambda$-modules to abelian groups (see \cite[10.2.35]{Preb2}). Hence $\pi^{k_0}\Ext^1(L,-)$ is also a finitely presented functor.  As $\Tf_\Lambda$ is the smallest definable
subcategory of $\Lambda$-modules containing all $\Lambda$-lattices, it follows that $\pi^{k_0}\Ext^1(L,N)=0$ for every $\Lambda$-lattice $L$ and $N \in \Tf_\Lambda$.

Maranda's Theorem \cite[Theorem 30.14]{CR} says that, under the
previous assumptions, if $k$ is any integer $> k_0$, then any two $\Lambda$-lattices $M$, $N$
are isomorphic over $\Lambda$ if and only if $M_k$ and $N_k$ are isomorphic over $\Lambda_k$.
Moreover, even decomposability lifts from $M_k$ to $M$ \cite[Theorem 30.19]{CR}.

A generalization of these results to pure injective $\Lambda$-modules is given in a
parallel paper \cite{Gre}, where the following is shown.

\begin{theorem}\label{maranda}
Let $N$, $N'$ be $R$-reduced $R$-torsionfree pure injective $\Lambda$-modules. If $N_k \simeq
N'_k$ for some integer $k > k_0$, then $N \simeq N'$. Moreover, if $N$ is indecomposable over
$\Lambda$, then the same is true of $N_k$ over $\Lambda_k$ for every $k > k_0$.
\end{theorem}

On the other hand, let us also mention, again from \cite{Gre}:

\begin{theorem}\label{ce}
There exists a module $N$ over $\widehat{\Z_2} C(2)^2$ such that $N$ is torsionfree and reduced
over $\Z_2$, $N_k$ is pure injective for all positive integer $k$ but $N$ is not pure-injective.
\end{theorem}

We propose here some applications of the classical Maranda Theorem to our framework.
If $f : M \rightarrow N$ is a
morphism of $\Lambda$-modules, then we will write $\overline{f}$ for the induced homomorphism from $M_k$ to $N_k$ ($k$ a positive integer). The following lemma is implicit in the proof of Maranda Theorem in \cite{CR}.

\begin{lemma}\label{onelatt}
Let $L$ be a $\Lambda$-lattice and $N$ an $R$-torsionfree $\Lambda$-module. If $k > k_0$ then for all $g \in \Hom_{\Lambda_k} (L_k, N_k)$ there exists $f\in \Hom_\Lambda(L, N)$ such that for all $m\in L$, $\pi^{k-k_0}+\Lambda\pi^k\, | \, (f(m)+ \pi^k M) -g (m + \pi^k M)$.
\end{lemma}
%\begin{proof}
%As said, the proof of Maranda's theorem in \cite[30.14]{CR} shows that for all
%$g\in \Hom_{\Lambda_k}(L_k,M_k)$ there exists $f\in \Hom_\Lambda(L,M)$ such that for all $m\in L$, %$\pi^{k-%k_0}|\overline{f(m)}-g(\overline{m})$ whenever $L$ is a $\Lambda$-lattice and $M\in \Tf_\Lambda$.
%\end{proof}

\begin{definition}
{\rm Assume $k > k_0$. Let $\phi$ be a pp-formula of $L_{\Lambda}$ in $l$ free variables $\mathbf{x} = (x_1, \ldots,
x_l)$.
Suppose that $(M, \mathbf{m})$ is a free realisation of $\phi$, where $M$ is $R$-torsionfree,
so a $\Lambda$-lattice, and that $\phi \geq \bigwedge_{i=1}^l \pi^{k-k_0} \mid x_i$. Then let $\phi_k$ be the generator of the pp-type of $\mathbf{m}+\pi^kM \in M_k$.}
\end{definition}

Thus $\phi_k$ is a pp-formula of $L_{\Lambda_k}$ with $l$ free variables.

Note that $\phi_k$ is well defined. In fact, let $(M, \mathbf{m})$ and $(N, \mathbf{n})$
be free realisations of $\phi$,
with both $M$ and $N$ $\Lambda$-lattices. Then there exist $\Lambda$-module morphisms
$f : M \rightarrow N$ and $g: N \rightarrow M$ such that $f(\mathbf{m})=\mathbf{n}$ and $g(\mathbf{n}) = \mathbf{m}$. The homomorphisms $\overline{f} : M_k\rightarrow N_k$ and $\overline{g} : N_k\rightarrow M_k$ induced by $f$ and $g$ respectively are such that $\overline{f}(\mathbf{m}+\pi^kM)=\mathbf{n} + \pi^k N$ and $\overline{g}(\mathbf{n} + \pi^k N)=\mathbf{m}
+ \pi^kM$. Thus the pp-type of $\mathbf{n}+\pi^k N$ in $N_k$ is equal to the pp-type of $\mathbf{m}+ \pi^k M$ in $M_k$, which guarantees that the above pp-formula $\phi_k$ is well defined, as said.

\begin{lemma}\label{boh}
Let $k > k_0$. Suppose that $\phi\in\pp^l_{\Lambda}$ is freely realised in a $\Lambda$-lattice and $\phi\geq \bigwedge_{i=1}^l\pi^{k-k_0}|x_i$. Then, for all $R$-torsionfree $\Lambda$-modules $N$
and $l$-tuples $\mathbf{n}\in N$,
\[\mathbf{n} \in \phi(N) \text{ if and only if } \mathbf{n} +\pi^kN\in \phi_k(N_k).\]
\end{lemma}

\begin{proof}
Let $M$ be a $\Lambda$-lattice and suppose that $\phi$ is freely realised by $\mathbf{m}\in M$. Then, by definition, $\phi_k$ is freely realised by $\mathbf{m}+\pi^kM$ in $M_k$.

Let $N$ be an $R$-torsionfree $\Lambda$-module. If $\mathbf{n} \in \phi(N)$ then there exist
a morphism $f : M \rightarrow N$ sending $\mathbf{m}$ to $\mathbf{n}$. Consequently
$\overline{f} : M_k \rightarrow N_k$ sends $\mathbf{m}+ \pi^kM$ to $\mathbf{n}+\pi^kN$. Therefore $\mathbf{n}+ \pi^k N \in \phi_k(N_k)$.

Now suppose that $\mathbf{n} + \pi^kN \in \phi_k(N_k)$. Then there exists a morphism $f : M_k
\rightarrow N_k$ sending $\mathbf{m}+\pi^kM$ to $\mathbf{n}+\pi^kN$. By \ref{onelatt}, there exists $g\in \Hom_{\Lambda}(M,N)$ such that $g(m_i)-n_i\in \pi^{k-k_0}N$ for all $i = 1, \ldots, l$.
Thus $\mathbf{n}$ satisfies the pp-formula $\bigwedge_{i=1}^l\pi^{k-k_0} \mid x_i + \phi(\mathbf{x})$.
Since $k-k_0\geq 1$ and $\phi\geq \bigwedge_{i=1}^l\pi^{k-k_0} \mid x_i$, $\mathbf{n}\in\phi(N)$.
\end{proof}

Given $\phi\in\pp_{\Lambda_k}^n$, we define $\phi^*\in\pp_{\Lambda}^n$ such that for all $N\in\Mod\text{-}\Lambda$, $\textbf{m}\in \phi^*(N)$ if and only if $\textbf{m}+\pi^{k}N\in \phi(N_k)$. Suppose $\phi\doteq \exists \mathbf{y} \ (\mathbf{x}\mathbf{y})T=0$ where $T$ is an appropriately sized matrix with entries from $\Lambda_k$. Further suppose that $T_{ij}:=t_{ij}+\Lambda\pi^k$ where $t_{ij}\in\Lambda$. Let $T^*$ be the matrix with entries $t_{ij}$ and let $\phi^*:=\exists\mathbf{y} \ \pi^k|(\mathbf{x}\mathbf{y})T^*$. A quick computation shows that for all $N\in\Mod\text{-}\Lambda$, $\mathbf{m}\in\phi^*(N)$ if and only if $\mathbf{m}+N\pi^k\in\phi(N_k)$ as required. Using this property of $\phi^*$, one can check that for all $N\in\Mod\text{-}\Lambda$, the map which sends $\phi(N_k)\in\pp_{\Lambda_k}^n(N_k)$ to $\phi^*(N)\in \pp_\Lambda^n(N)$ is a lattice homomorphism. Note that $(\pi^{k-k_0}+\Lambda\pi^k|\mathbf{x})^*$ is equivalent to $\pi^{k-k_0}|\mathbf{x}$.

The next proposition applies to arbitrary $R$-torsionfree $\Lambda$-modules.
In its statement we write $\pi^{k-k_0}|\mathbf{x}$ to mean the pp-$n$-formula $\bigwedge_{i=1}^n\pi^{k-k_0}|x_i$. Recall that, for
every pp-formula $\phi(x)\in\pp^n_\Lambda$, $\overline{\phi}(x)$
is the pp-formula associated to $\phi(x)$ defined just before Lemma \ref{overline}.

\begin{prop}\label{added}
Let $k>k_0$. The map from the closed interval $[\pi^{k-k_0}|\mathbf{x},\mathbf{x}=\mathbf{x}]$ of $\pp_\Lambda^n$ to the interval $[\pi^{k-k_0}|\mathbf{x},\mathbf{x}=\mathbf{x}]$ of $\pp_{\Lambda_k}^n$ which sends any $\phi\in [\pi^{k-k_0}|\mathbf{x},\mathbf{x}=\mathbf{x}]$ to $\overline{\phi}_k\in[\pi^{k-k_0}|\mathbf{x},\mathbf{x}=\mathbf{x}]$ induces a lattice isomorphism from $[\pi^{k-k_0}|\mathbf{x},\mathbf{x}=\mathbf{x}]_N$ to $[\pi^{k-k_0}+\pi^k\Lambda |\mathbf{x}, \mathbf{x}=\mathbf{x}]_{N_k}$ for all $N\in \Tf_\Lambda$.

In particular, this lattice isomorphism is inverse to the lattice homomorphism which sends $\psi(N_k)\in [\pi^{k-k_0}+\pi^k\Lambda |\mathbf{x}, \mathbf{x}=\mathbf{x}]_{N_k}$ to $\psi^*(N)\in [\pi^{k-k_0}|\mathbf{x},\mathbf{x}=\mathbf{x}]_N$ for all $N\in\Tf_\Lambda$.
\end{prop}

\begin{proof}
First note that $\pi^{k-k_0}|\mathbf{x}$ is freely realised by the n-tuple from $\Lambda$ with all entries $\pi^{k-k_0}$. Thus $\pi^{k-k_0}|\mathbf{x}\leq \phi$ implies $\pi^{k-k_0}|\mathbf{x}\leq \overline{\phi}$. So $\overline{\phi}_k$ is defined. Suppose $\pi^{k-k_0}|\mathbf{x}\leq \phi$. Then $(\pi^{k-k_0},\ldots,\pi^{k-k_0})\in\phi(\Lambda)=\overline{\phi}(\Lambda)$. So $(\pi^{k-k_0}+\Lambda\pi^k,\ldots,\pi^{k-k_0}+\Lambda\pi^k)\in\overline{\phi}_k(\Lambda_k)$. Hence $\pi^{k-k_0}+\Lambda\pi^k|\mathbf{x}\leq \overline{\phi}_k$.

Recall that $(\pi^{k-k_0}+\Lambda\pi^k|\mathbf{x})^*$ is equivalent to $\pi^{k-k_0}|\mathbf{x}$. So $\psi(N_k)\mapsto \psi^*(N)$ defines a lattice homomorphism from $[\pi^{k-k_0}+\pi^k\Lambda |\mathbf{x}, \mathbf{x}=\mathbf{x}]_{N_k}$ to $[\pi^{k-k_0}|\mathbf{x},\mathbf{x}=\mathbf{x}]_N$ for all $N\in\Tf_\Lambda$.

Since, when it exists, the set-wise inverse of a lattice homomorphism is a lattice isomorphism, it is therefore enough to show that for all $N\in\Tf_\Lambda$, $\phi\in[\pi^{k-k_0}|\mathbf{x},\mathbf{x}=\mathbf{x}]$ and $\psi\in [\pi^{k-k_0}+\pi^k\Lambda|\mathbf{x}, \mathbf{x}=\mathbf{x}]$, $(\overline{\phi}_k)^*(N)=\phi(N)$ and $\overline{(\psi^*)}_k(N_k)=\psi(N_k)$. But this follows from Lemma \ref{boh} and the property of $\phi^*$ described just before this proposition.
\end{proof}

\section{Finite lattice representation type}\label{S-finite}

In this final section we recover our largest setting and we deal
with a Dedekind domain $R$ which is not a field, with its field of fractions $Q$
and with an $R$-order $\Lambda$ in a finite dimensional $Q$-algebra $A$.

Recall that $\Lambda$ is said to be of {\sl finite lattice representation type} if
it has only finitely many non isomorphic indecomposable lattices. Our aim is
to obtain a complete description of $\Zg_\Lambda^{tf}$ when $\Lambda$ is of
finite lattice representation type.

But let us first concern ourselves with the m-dimension of $\pp_\Lambda^1 (\Tf_\Lambda)$ under
the finite lattice representation type hypothesis.

\begin{prop}\label{md3}
Let $\Lambda$ be an order over a Dedekind domain $R$. If $\Lambda$ is
finite lattice representation
type then $\text{m-dim} \, \pp_\Lambda^1 (\Tf_\Lambda)=1$.
\end{prop}

\begin{proof}
Let $L_1,\ldots,L_n$ be a complete list of indecomposable $\Lambda$-lattices up to isomorphism. By Fact \ref{oplus} and Proposition \ref{lattpi}, the canonical surjection from $\pp_\Lambda^1 (\Tf_\Lambda)$ to $\pp_\Lambda^1 (\oplus_{i=1}^nL_i)$ is an isomorphism. Since $\oplus_{i=1}^nL_i$ is a $\Lambda$-lattice, by Corollary \ref{md2}, $\pp_\Lambda^1 (\oplus_{i=1}^nL_i)$ has m-dimension $1$.

%Since all pp-definable subgroups of $\oplus_{i=1}^nL_i$ are $R$-submodules, $\pp_\Lambda^1 (\oplus_{i=1}^nL_i)$ is a sublattice of $\text{ Sub}_R (\oplus_{i=1}^nL_i)$.
%But $\oplus_{i=1}^nL_i$ is a $\Lambda$-lattice, whence by Corollary \ref{md2}
%$\text{ Sub}_R (\oplus_{i=1}^nL_i)$
%has m-dimension $1$. Thus $\pp_\Lambda^1 (\oplus_{i=1}^nL_i)$ has m-dimension either
%$1$ or $0$.
%
%We now show that $\pp_\Lambda^1(\Lambda)$ is not of finite length and
%consequently $\pp_\Lambda^1 (\Tf_\Lambda)$ is not of finite length (whence it has m-dimension 1).
%Take $c\in R$ a non-unit. Since $\Lambda$ is finitely generated and torsionfree over $R$,
%$\Lambda$ is over $R$ a direct summand of $R^l$ for some $l\in\N$. For any non-zero $m\in R^l$
%there exists $n\in\N$ such that $c^n$ does not divide $m$. Thus for any non-zero $m\in\Lambda$
%there exists $n\in\N$ such that $c^n$ does not divide $m$. Thus there exists $m\in\Lambda$ such that $c$ does not divide $m$. Since $\Lambda$ is $R$-torsionfree, for all $n\in\N$ there exists $m\in\Lambda$ such that $c^n|m$ but $c^{n+1}$ does not divide $m$. Thus $\pp_\Lambda^1 (\Lambda)$ has a strictly decreasing chain and therefore is not of finite length.
\end{proof}

%\begin{remark}
%If $T$ is a (quasi-)compact topological space and the CB-rank of $T$ is one
%then all points in $T$ are either %closed or isolated.
%\end{remark}

Here is a first consequence of this proposition.
Let $S(\Lambda)$ be the set of maximal ideals $P$ of $R$ such that $\Lambda_P$ is not maximal. If $Q\Lambda$ is separable then $S(\Lambda)$ is finite, see \cite[p. 642]{CR}. Moreover, see \cite[11.5]{MO}, $\Lambda_P$ is maximal if and only if $\widehat{\Lambda_P}$ is maximal. By \cite[18.1]{MO}, if $\widehat{\Lambda_P}$ is maximal then $\widehat{\Lambda_P}$ is hereditary. So, by \cite[10.6]{MO}, all $\widehat{\Lambda_P}$-lattices are projective. Finally, since that category of $\widehat{\Lambda_P}$-lattices is Krull-Schmidt, there are only finitely many indecomposable projective lattices and hence $\widehat{\Lambda_P}$ is finite lattice representation type.

\begin{cor}\label{md4}
Suppose that $Q\Lambda$ is separable and $S(P)$ non-empty. The m-dimension of $\pp_{\Lambda}^1(\Tf_\Lambda)$ is equal to $\text{max}_{P\in S(\Lambda)}\pp_{\widehat{\Lambda_P}}^1(\Tf_{\widehat{\Lambda_P}})$.
\end{cor}
\begin{proof}
By \ref{md3}, if $P\notin S(\Lambda)$ then $\pp_{\widehat{\Lambda_P}}^1(\Tf_{\widehat{\Lambda_P}})$ has m-dimension $1$. For any maximal ideal $P$ of $R$, the m-dimension of $\pp_{\widehat{\Lambda_P}}^1(\Tf_{\widehat{\Lambda_P}})$ is greater than or equal to $1$.
\end{proof}

Note that, if $S(P)$ is empty, then the m-dimension is 1.

The following is a  further consequence of Proposition \ref{md3}: if $\Lambda$ is of finite lattice representation type, then $\Zg_\Lambda^{tf}$ has the isolation property and so its Cantor-Bendixson
rank is equal to m-dim $pp_\Lambda^1 (\Tf_\Lambda)$, namely it is 1 (see \cite[Proposition 5.3.17]{Preb2}).
Let us confirm this conclusion in a different way, also providing the topological description of
$\Zg_\Lambda^{tf}$ we promised before. For this we focus on completions of localizations of $\Lambda$
at maximal ideals of $R$, so at complete discrete valuation domains. First an easy fact - a sort of
converse of the main result we are going to prove.

\begin{prop}\label{cblatt}
Let $R$ be a complete discrete valuation domain and $\Lambda$ an order over $R$. If
$\Zg_\Lambda^{tf}$ has Cantor-Bendixson rank $1$ then $\Lambda$ is finite lattice representation
type.
\end{prop}
\begin{proof}
Let $\pi$ generate the maximal ideal of $R$.
If $N \in \left(x=x/ \pi \mid x \right)$ then $N$ is not $\pi$-divisible and hence by Corollary \ref{closedina}
is not closed. Therefore $\left(x=x / \pi \mid x\right)$ contains only isolated points. Since $\left(x=x / \pi \mid
x \right)$ is compact, it must be finite.
Since all indecomposable $\Lambda$-lattices are pure injective and not $\pi$-divisible, $\Lambda$ is
of finite lattice representation type.
\end{proof}

Next we provide the description of $\Zg_\Lambda^{tf}$ when $\Lambda$ is an order over a complete
discrete valuation domain.

\begin{prop}\label{flt1}
Suppose that $R$ is a complete discrete valuation domain with field of fractions $Q$,
$\Lambda$ is an order over $R$, $A = Q\Lambda$ is a semisimple $Q$-algebra
and $\Lambda$ is of finite lattice representation type. Then the set  $\Zg_\Lambda^{tf}$
consists exactly of
\begin{itemize}
\item finitely many indecomposable lattices over $\Lambda$,
\item finitely many simple $A$-modules.
\end{itemize}
The indecomposable modules over $A$ are closed points. If $N$ is an indecomposable lattice then
a simple $A$-module $M$ is in the closure of $N$ if and only if $M$ is a direct  summand of $Q N$.
\end{prop}
\begin{proof}
By Proposition \ref{lattpi}, the set of indecomposable $\Lambda$-lattices is dense in $\Zg_\Lambda^{tf}$. Thus if $N_1,\ldots,N_n$ are the indecomposable lattices over $\Lambda$, then $\Zg_\Lambda^{tf}$ coincides with the closure of $\{N_1,\ldots N_n\}$ and hence of the union of the closures of the $N_i$ with $i = 1, \ldots, n$.

In Proposition \ref{closure} we showed that any point in the closure of $N$ an indecomposable lattice which is not equal to $N$ is a closed point. By Corollary \ref{closedina}, any closed point is an $A$-module. Thus we have shown that $\Zg_\Lambda^{tf}$ has exactly the points stated in the proposition.

The description of the topology follows from Lemmas \ref{close1} and \ref{lattiso}.
\end{proof}

\begin{cor}\label{nonfg}
Assume $R$ and $\Lambda$ as before, hence in particular $\Lambda$ of finite
lattice representation type. Let $p$ be a non finitely generated indecomposable
pp-type in the theory $T_\Lambda^{tf}$ of $R$-torsionfree $\Lambda$-modules.
Then $p$ contains all divisibility formulas $\pi^k \mid x$ for $k$ a positive integer.
\end{cor}

\begin{proof}
Any element of a simple $A$-module realizes all these formulas.
\end{proof}

Note that Herzog and Puninskaya verified a similar result for torsionfree modules
over 1-dimensional commutative noetherian local complete domains, see
\cite[Theorem 6.6]{HP}.

Now let us come back to a Dedekind domain $R$ and to an $R$-order $\Lambda$
in a separable $Q$-algebra $A$. The following hold:
\begin{itemize}
\item (\cite[Exercise 4.7 p. 99]{CR} there are only finitely many maximal ideals
$P$ of $R$ such that $\widehat{\Lambda_P}$ is not a maximal order;
\item (\cite[Proposition 33.1]{CR} if $\widehat{\Lambda_P}$ is maximal, then
it is of finite lattice representation type, and hence the topology of
$\Zg_{\widehat{\Lambda_P}}^{tf}$ is that described in Proposition \ref{md3}.
\end{itemize}

Recall that this topology is the same both in $\Zg_{\widehat{\Lambda_P}}^{tf}$
and $\Zg^{tf}_{\Lambda_P}$, when restricted to its $R$-reduced part (Theorem \ref{top}).

By the proof of \cite[Theorem 33.2]{CR}, if $\Lambda$ is of finite lattice representation type
then each non maximal $\Lambda_P$ is.

On this basis it is easy to deduce:

\begin{theorem}\label{end}
Let $R$ be a Dedekind domain and
$\Lambda$ an $R$-order in a separable $Q$-algebra $A$. Assume $\Lambda$
of finite lattice representation type. Then the Cantor-Bendixson rank of
$\Zg_\Lambda^{tf}$ is 1, and
\begin{itemize}
\item the isolated points are the indecomposable
$\widehat{\Lambda_P}$-lattices, where $P$ ranges over maximal ideals of $R$,
\item the points of Cantor-Bendixson rank 1 are the simple $A$-modules.
\end{itemize}
\end{theorem}

Let us give some examples illustrating the previous results. The first, over a complete
discrete valuation domain $R$, was proposed by Gena Puninski. Indeed it is one of
the last suggestions he left to us. So we like to mention it as a tribute to his
memory.

\begin{exam}\label{final}
Let $R$ be as just said, $\pi$ be a generator of its maximal ideal.
Let $\Lambda = \bsm R & R\\ \pi^2R & R \esm$ (see
\cite[p. 779]{CR}. Also, let $e_1$, $e_2$ denote for simplicity the
idempotents $\bsm 1 & 0\\ 0 & 0 \esm$ and $\bsm 0 & 0\\ 0 & 1 \esm$, respectively.
It is well known that $\Lambda$ has finite lattice representation type. In fact $\Lambda$ is
Gorenstein, i.e.
projective modules $P_1= e_1 \Lambda = \bsm R & R\\ 0 & 0 \esm$ (so basically
$(R, R)$) and $P_2 = e_2 \Lambda = \bsm 0 & 0 \\ \pi^2R & R \esm$ (hence  $(\pi^2, R)$)
are injective (in the category of lattices). The only remaining indecomposable lattice is $P= (\pi R, R)$
(note that $\bsm \pi R & R\\ 0 & 0 \esm$ and $\bsm 0 & 0 \\ \pi R & R \esm$ are isomorphic
as $\Lambda$-modules).

Hence a description of $\Zg^{tf}_\Lambda$ follows from Proposition \ref{flt1}. Anyway
let us follow Gena's approach for the reasons we said.

First of all, note that, being Gorenstein, $\Lambda$ admits a unique overorder
$\Lambda'= \bsm R& R\\ \pi R & R\esm$ which is hereditary;
and $P$ is defined over $\Lambda'$, i.e. $\Lambda'$ is the ring of definable scalars of $P$.
Furthermore the following is the $\AR$-quiver of $\Lambda$

\begin{center}
$
\xymatrix@C=14pt@R=8pt{%
&*+{_{P_1}}\ar[dr]^{\pi}&\\
*+{_P}\ar[ur]^{\subset}\ar[dr]_{\pi}&&*+{_P}\\
&*+{_{P_2}}\ar[ur]_{\subset}&\\
}
$
\end{center}

\noindent
where $\pi$ denotes the multiplication by $\pi$. From that we can see irreducible
morphisms in the category of lattices and the unique almost split sequence:
$$
0\to P\xr{(i, \pi)} P_1\oplus P_2 \xr{\bsm\pi\\-i\esm} P\to 0\, ,
$$
where $i$ denotes inclusion. In detail the two intermediate morphisms act as follows:
\begin{itemize}
\item for all $a, b \in R$, $(i, \pi)$ maps $(\pi a, b)$ to $( (\pi a, b), (\pi^2 a, \pi b))$,
\item for all $a', b', c', d' \in R$, $\bsm \pi \\ -i \esm$ sends
$((a', b'), (\pi^2 c', d'))$ to $(\pi a' - \pi^2 c', \pi b' - d')$.
\end{itemize}

Let $N$ be an indecomposable $R$-torsionfree pure injective
$\Lambda$-module. First suppose that there exists $0\neq n\in Ne_1$.
Hence look at pointed indecomposable lattices $(M,m)$
such that $m\in M e_1$. Up to equivalence (of types realized by $m$)
here is a complete list of them:
\begin{itemize}
\item $(P_1, (\pi^k, 0))$, $\ k\geq 0$,
\item $(P_2, (\pi^l, 0))$,  $l\geq 2$,
\item $(P, (\pi^m,0))$, $m\geq 1$.
\end{itemize}

Furthermore the following is the pattern of the module $(P_1,0)$, i.e., the poset of morphisms from $P_1$ to indecomposable lattices (see \cite{Pu} for a definition).
Here we use an \lq \lq exponential" notation: for instance $(P_1, k)$ abbreviates
$(P_1, (\pi^k,0))$.

\vspace{2mm}

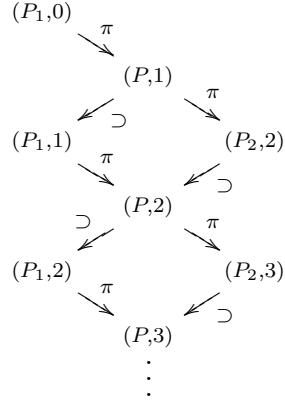
\begin{center}
$
\xymatrix@C=10pt@R=8pt{%
*+{_{(P_1,0)}}\ar[dr]^{\pi}&&\\
&*+{_{(P,1)}}\ar[dl]^(.4){\supset}\ar[dr]^{\pi}&&\\
*+{_{(P_1,1)}}\ar[dr]^{\pi}&&*+{_{(P_2,2)}}\ar[dl]^(.4){\supset}\\
&*+{_{(P,2)}}\ar[dl]_{\supset}\ar[dr]^{\pi}&\\
*+{_{(P_1,2)}}\ar[dr]^{\pi}&&*+{_{(P_2,3)}}\ar[dl]^(.4){\supset}\\
&*+{_{(P,3)}}\ar@{}+<0pt,-10pt>*{.}\ar@{}+<0pt,-16pt>*{.}\ar@{}+<0pt,-22pt>*{.}&\\
}
$
\end{center}

\vspace{3mm}

We easily derive that there is a unique (critical over zero) indecomposable non finitely generated
type $p$ in the interval
$[x=0, \, e_1\mid x]$ in $pp^1_\Lambda (\Tf_\Lambda)$.
Furthermore $p$ is realized by $(1,0)\in (Q,Q)= S$, the simple module over $A= M_2(Q)$.
Thus in this case $N\cong P_i, P$ or $N\cong S$.

It remains to consider the case when there exists $0\neq n\in Ne_2$. Again look at indecomposable pointed lattices
$(M,m)$ where $m\in Me_2$. They form the following pattern, where $(P_1, k)$ now
abbreviates $(P_1, (0,\pi^k))$, and so on.

\vspace{2mm}

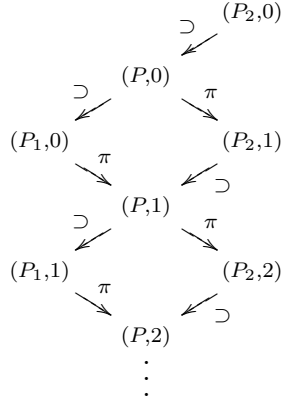
\begin{center}
$
\xymatrix@C=10pt@R=8pt{%
&&*+{_{(P_2,0)}}\ar[dl]_{\supset}\\
&*+{_{(P,0)}}\ar[dl]_{\supset}\ar[dr]^{\pi}&\\
*+{_{(P_1,0)}}\ar[dr]^{\pi}&&*+{_{(P_2,1)}}\ar[dl]^(.4){\supset}\\
&*+{_{(P,1)}}\ar[dl]_{\supset}\ar[dr]^{\pi}&\\
*+{_{(P_1,1)}}\ar[dr]^{\pi}&&*+{_{(P_2,2)}}\ar[dl]^(.4){\supset}\\
&*+{_{(P,2)}}\ar@{}+<0pt,-10pt>*{.}\ar@{}+<0pt,-16pt>*{.}\ar@{}+<0pt,-22pt>*{.}&\\
}
$
\end{center}

\vspace{3mm}

Then there is a unique indecomposable non finitely generated
type $q$ in the interval $[x=0, e_2\mid x]$ which is realized as
$(0,1)\in (Q,Q)= S$. Again we conclude that $N\cong P_i, P$ or $N\cong S$.

This completes the description of the $R$-torsionfree part of the Ziegler spectrum of our ring.
\end{exam}

\begin{exam}\label{pcycl}
The second example concerns an integral group ring $\Z \, C(p)$ with $p$ a
prime. This is of finite lattice representation type (see \cite[33A, in particular
p. 690, and 34B]{CR}), whence Theorem \ref{end} applies. Here is an explicit
description of the torsionfree part of the Ziegler spectrum. Let $g$ denote
a generator of $C(p)$ and $\zeta_p$ a primitive $p$-th root of 1 in $\C$, so
a root of the cyclotomic polynomial $\Phi_p(t) = t^{p-1} + t^{p-2} + \ldots + t + 1 \in \Z[t]$.
Incidentally, let $\Phi_1 (t) = t - 1$, whence $t^p - 1 = \Phi_1 (t) \cdot \Phi_p(t)$.
Also, let $e_1 = {1 \over p} \Phi_p (g)$ and $e_2 = 1 - e_1$ be the primitive idempotents
of the algebra $\Q \, C(p)$. Thus the points of $\Zg^{tf}_{\Z \, C(p)}$ are the following.
\begin{itemize}
\item First of all, three isolated lattices over $\Lambda = \widehat{\Z_p} \, C(p)$, i.e.,
$\Lambda e_i$, $i = 1, 2$, and $\Lambda$ itself, in other words $\widehat{\Z_p}$,
$\widehat{\Z_p} (\zeta_p)$ and $\widehat{\Z_p} \, C(p)$. In the first two
cases $g$ acts as the identity and the multiplication by $\zeta_p$, respectively.
Moreover $\widehat{\Z_p} \, C(p)$ corresponds to the pullback
of $\widehat{\Z_p}_p$ and $\widehat{\Z_p} (\zeta_p)$ via the projections onto $\Z/ p \Z$
sending 1 and $\zeta_p$ into $1 + p \Z$ (see \cite{Le}).
\item Next, for every prime $q \neq p$, two more isolated points, $\widehat{\Z_q}$
and $\widehat{\Z_q}  (\zeta_p)$ respectively, as now $\widehat{\Z_q} \, C(p)$ is their direct sum.
\item Finally, two more points of Cantor-Bendixson rank 1, $\Q$ and $\Q(\zeta_p)$, as $\Q \,
C(p)$ is again their direct sum.
\end{itemize}
The topology is also easy to describe.
\begin{itemize}
\item In fact $\Lambda e_1 = \widehat{\Z_p}$, $\Lambda e_2 = \widehat{\Z_p} \, (\zeta_p)$
and $\Lambda = \widehat{\Z_p} \, C(p)$  are the only points in the basic open set $(x = x / p \mid x)$,
and indeed can be isolated, and separated from each other, as follows: $\Lambda$ by
$(\sum_{j < p} g^j  \mid x \, / \, p \mid x)$ where $\sum_{j < p} g^j = p e_1$ (see
\cite[end of p. 57]{Bu1}) and $\Lambda e_1$, $\Lambda e_2$ by
$(e_1 \mid x \, / \, p e_1 \mid x)$, $(e_2 \mid x \, / \, (1 - \zeta_p) e_2 \mid x)$ respectively
(see \cite[proof of 3.3(a)]{Bu1}). Note that, properly speaking, $e_1$ and $e_2$ are not in
$\Z \, C(p)$. However, as we are working in a $\Z$-torsionfree framework we can use here
the simple trick of multiplying every involved scalar by $p$ and expressing the previous
open sets as $(p e_1 \mid x \, / \, p^2 e_1 \mid x)$ and $(p e_2 \mid x \, / \, p (1 - \zeta_p) e_2
\mid x)$.
%$((1 - g + \ldots g^{p-1}) \mid x \, / \, x = 0)$ and $(x(1-g) = 0 \, / \,
%(1 + g + \ldots g^{p-1}) \mid x)$, $\widehat{\Z_p} (\zeta_p)$ by that of $(x = x / p \mid x)$,
%$((1 -g) \mid x \, / \, x = 0)$ and $(x(1+g+ \ldots + g^{p-1}) = 0 \, / \,
%(1 -g) \mid x$, and finally $\widehat{\Z_p} \, C(p)$ by that of
%$(x = x / p \mid x)$, $((1 + g + \ldots g^{p-1}) \mid x  \, / \, x = 0)$ and $(x(1-g) \mid x \, / \,
%x = 0)$;
\item For every prime $q \neq p$, $\widehat{\Z_q}$ and $\widehat{\Z_q}  (\zeta_p)$ are
isolated from the other points by $(x=x \, / \, q \mid x)$ and indeed separated from each other,
and hence isolated at all, by
$((1+g+ \ldots + g^{p-1}) \mid x \, / \, q \mid x )$  and $( (1-g) \mid x \, / \, q \mid x)$, respectively.
\item $\Q$ and $\Q(\zeta_p)$, that is, $\Q \Lambda e_1$ and $\Q \Lambda e_2$, are the points
of Cantor-Bendixson rank 1 and at this level can be separated from each other, for instance, by
$(x (1 - e_1) p = 0 \, /  \, x = 0)$ and $(x (1 - e_2) p = 0 \, /  \, x = 0)$, respectively.

%any pair $((1 + g + \ldots g^{p-1} \mid x \land \bigwedge_{n \leq N} p_n \mid x) \, / \, x = 0)$ and $((1-g) %\mid x \land \bigwedge_{n \leq N} p_n \mid x) \, / \, x= 0)$ where
%$p_n$ ($n \in \N$) is the sequence of primes, or also of primes $\neq p$. All these open sets
%are compact. Hence the former ones (together with their unions with the open sets isolating
%$\widehat{\Z_p} C(p)$-lattices) are a basis of neighbourhoods of $\Q$, and the latter ones
%(with a similar integration) of $\Q (\zeta_p)$.
\end{itemize}
\end{exam}

\begin{exam}\label{p2cycl}
Finally let us deal with the integral group ring $\Z \, C(p^2)$ with $p$ a prime. This is
again a $\Z$-order of finite lattice representation type. A description of $\Zg^{tf}_{\Z \, C(p^2)}$,
both points and topology, can be extracted from the classification of lattices over $\widehat{\Z_p} C(p^2)$ given in \cite[34C p. 730]{CR} in terms of extension groups, or in \cite{Ka} in terms of pullbacks, or also in
\cite[$\S$ 4]{Bu1}. We follow this third approach. Let $g$ still denote a generator of the
group $C(p^2)$, $e_1$, $e_2$, $e_3$ be the primitive idempotents of the algebra $\Q \, C(p^2)$.
Thus
$$
e_1 = {1 \over {p^2}} \sum_{j< p^2} g^j = {1 \over {p^2}} \Phi_p (g) \Phi_p (g^p),
$$
$$
e_2 = {1 \over {p^2}} (p - \Phi_p(g)) \Phi_p(g^p), \, \, e_3 = {1 \over p} (p - \Phi_p (g^p))
$$
where $\Phi_p(t^p) = \Phi_{p^2} (t)$ is the cyclotomic polynomial of order $p^2$. Then the points
of $\Zg^{tf}_{\Z \, C(p^2)}$ are the following.
\begin{itemize}
\item Let us start this time from simple $\Q \, C(p^2)$-modules, that is, from points of Cantor-Bendixson rank 1. They are $\Q$, $\Q(\zeta_p)$ and $\Q (\zeta_{p^2})$ where $\zeta_p$ and $\zeta_{p^2}$ are primitive roots of 1
of order $p$, $p^2$ respectively.
\item When $q$ is a prime different from $p$, $\widehat{\Z_q} \, C(p^2)$-lattices admit a similar description.
\item Hence let us focus on $\Lambda = \widehat{\Z_p} \, C(p^2)$. Indecomposable lattices are now $4p+1$.
The first four are $\Lambda$ itself and the $\Lambda e_i$, $i = 1, 2, 3$, that is, $\widehat{\Z_p} \, C(p^2)$ and then
$\widehat{\Z_p}$, $\widehat{\Z_p} (\zeta_p)$ and $\widehat{\Z_p} (\zeta_{p^2})$. In the three last cases
$g$ acts as the multiplication by 1, $\zeta_p$ and $\zeta_{p^2}$ respectively. The remaining $4p - 3$
correspond via the representation equivalence described in \cite[$\S \S$ 3 and 4]{Bu1} to the indecomposable
objects in the category of finite dimensional $\Z / p \Z$-representations of the directed Dynkin diagram $D_{2p}$.
These can be viewed as tuples $(W, (W_{0h}^s)_{0 < s < p, h = 1, 2}, W_1, W_2)$, where
\begin{enumerate}
\item $W$ is a vector space over $\Z / p\Z$,
\item $W_0 = W_{01}^1$, $W_1$, $W_2$ are subspaces of $W$ such that the sum of
any two of them gives the whole $W$,
\item the $W_{0h}^s$ form an increasing chain of subspaces $W_0 = W_{01}^1 \subseteq W_{02}^1
\subseteq W_{01}^2 \subseteq W_{02}^2 \subseteq \ldots \subseteq W_{01}^{p-1} \subseteq W_{02}^{p-1} = W$.
\end{enumerate}
Moreover Butler's functor uniformly defines in a first order way every such representation in the associated
lattice via quotients of pp-subgroups. Then it suffices for our purposes to list these indecomposable representations of $D_{2p}$. In fact they recursively determine the corresponding lattices as a sort of ID card.
As said, they are $4p-3$. In all of them the dimension of $W$ is either 1 or 2. In the former case, i.e. in dimension 1, we meet
\begin{itemize}
\item[a)] 3 points where $W_0 = W$ (and hence $W_{0h}^s = W$ for every $h$ and $s$) and the pair
$(W_1, W_2)$ is one among $(W,0)$, $(0,  W)$ and $(W,W)$,
\item[b)] $2p-3$ points where $W_1 = W_2 = W$, $W_0 = 0$ and the $W_{0h}^s$ are constantly $0$
before some $s$ and $h$ ($h = 1$ when $s = 1$) and then become equal to $W$.
\end{itemize}
In the 2-dimensional case, we find $2p - 3$ additional points in which $W_0$, $W_1$, $W_2$ are 1-dimensional
subspaces such that $W$ is the sum of any two of them (and so the intersection of any two
of them is 0), and the $W_{0h}^s$ equal $W_0$ before
some $s$ and $h$ ($h = 1$ when $s = 1$) and then coincide with $W$.
\end{itemize}
Next let us see the topology, so how to separate isolated points from each other. The case of primes $q \neq p$
can be handled as for $C(p)$, with slight complications. For instance the open sets
isolating $\widehat{\Z_q}$, $\widehat{\Z_q} (\zeta_p)$ and $\widehat{\Z_q} (\zeta_{p^2})$
are now $((1 + g + \ldots + g^{p-1}) \mid x \land (1 + g^p + \ldots + g^{p(p-1)}) \mid x \, / \,
q \mid x)$,  $((1 - g) \mid x \land (1 + g^p + \ldots + g^{p(p-1)}) \mid x \, / \,
q \mid x)$, $((1 - g) \mid x \land (1 + g + \ldots + g^{p-1}) \mid x \, / \,
q \mid x)$ respectively.  Also the analysis of  simple $\Q C(p^2)$-modules is similar to that of
$C(p)$.

Hence let us deal with $q = p$ and with
indecomposable lattices over $\Lambda = \widehat{\Z_p} \, C(p^2)$, those in $(x = x \, / \, p \mid x)$.

The way to isolate $\Lambda$ and the $\Lambda e_i$ ($i = 1, 2, 3$) is the same as for $C(p)$,
by $(\sum_{j < p^2} g^j \mid x \, / \, p \mid x)$, $(e_1 \mid x \, / \, p e_1 \mid x)$, $(e_2 \mid x \,
/ \, (1 - \zeta_p) e_2 \mid x)$, $(e_3 \mid x \, / \, (1 - \zeta_{p^2}) e_3 \mid x)$ respectively.

The further $4p-3$ points are those in the open set $(p^2 \sum_{1 \leq i \leq 3} e_i \mid x \, / \, p^2 \mid x)$
(see the construction in \cite[$\S$ 3]{Bu1}).
To separate them from each other, we can look at the
associated representations of $D_{2p}$ as abelian structures in their own language, because these representations
are uniformly pp-definable without parameters in the corresponding lattices.  Let us write for simplicity $x \in W_0$, $x \in W_1$ and so on to denote the formulas admitting this interpretation in any given representation. Thus
\begin{itemize}
\item[a)] the first 3 points are isolated by $(x \in W_0 \land x \in W_1 \, / \, x \in W_2)$,
$(x \in W_0 \land x \in W_2 \, / \, x \in W_1)$ and $(x \in W_1 \land x \in W_2 \land x \in W_0 \, / \, x =0)$,
\item[b)] the following $2p-3$ are isolated by $(x \in W_1 \land x \in W_2 \land x \in W_{01}^{s+1} \, /
\, x \in W_{02}^s)$ or $(x \in W_1 \land x \in W_2 \land  x \in W_{02}^s \, / \, x \in W_{01}^s)$ for the
right $s$.
\end{itemize}
Similarly, the last $2p-3$ are isolated by
$( x \in W_{01}^{s+1} \, / \, x \in W_{02}^s + (x \in W_1 \land x \in W_2))$
or
$(x \in W_{02}^s \, / \, x \in W_{01}^s + (x \in W_1 \land x \in W_2 ))$ for the right $s$.

\end{exam}

On this basis, one easily deduces the following:

\begin{theorem}\label{decd}
For every prime $p$, the first order theories of both $\Z$-torsionfree $\Z \, C(p)$-modules
and $\Z$-torsionfree $\Z \, C(p^2)$-modules are decidable.
\end{theorem}

In fact, the descriptions of the torsionfree part of the Ziegler spectrum of $\Z \, C(p)$ and $\Z \, C(p^2)$ fit with the conditions of the decidability criterion in \cite[Theorem 9.4]{Zi}and \cite[Theorem
17.12]{Preb1}. In fact both an effective list of indecomposable pure injective modules $N$ (possibly through the related representations of $D_{2p}$) and,
for every such module, an effective list of basic open sets $(\phi / \psi)$ around $N$ were already provided. Furthermore straightforward calculations
recursively determine $\phi(N)/\psi(N)$ from $N$, $\phi$ and $\psi$.

Notice that Theorem \ref{decd} positively solves expectations in
the final lines of \cite{MPT0}. See also \cite{T1}.

%Now let us discuss the generalized Maranda Theorem in the case of $\Lambda =
%\widehat{\Z_p} \, C(p)$ and $\widehat{\Z_p} \, C(p^2)$, as promised at the end of $\S$
%\ref{S-maranda}. Recall that $k_0 = 1, 2$ for these orders, respectively.

%In the case of $C(p)$ we have seen
%that there are exactly three lattices over $\widehat{\Z_p} \, C(p)$. On the other hand
%the class of $\Z/ p^2 \Z \, C(p)$-lattices, even if tame, contains infinitely many examples
%(all finite and so pure injective). A description of them can be extracted from \cite[30C,
%pp. 637--641]{CR}, see also the second half of \cite{T2}.
%The case of $C(p^2)$ is even worse, because, as underlined in the same pages of
%\cite{CR}, the class $\Z / p^3 \Z C(p^2)$-lattices is of wild representation type.

\end{document}